\documentclass[a4paper,11pt,leqno]{article}

\usepackage{amsmath,amssymb,eucal,amsthm}
\usepackage{enumerate}

\usepackage{latexsym,bm}
\usepackage{mathrsfs,amsmath,amssymb,secdot}
\usepackage{amscd}

\def\aa{{d}}
\def\bb{B}
\def\ba{\frak{B}}
\def\bm{{B}}
\def\bbm{\mathbb{B}}

\def\cc{C}

\def\eett{\eta}
\def\kk{k_1,\ldots,k_r}
\def\km{-k_1,\ldots,-k_r}
\def\mm{m_1,\ldots,m_r}

\def\ppp{\widetilde{P}}
\def\Li{{\rm Li}}

\long\def\comment#1{}

\def\re{{\rm Re}}
\allowdisplaybreaks

\font\fivecy=wncyr5  \def\sa{{\hbox{\fivecy X}}}

\DeclareFontEncoding{OT2}{}{}
\DeclareFontSubstitution{OT2}{cmr}{m}{n}
\DeclareSymbolFont{cyss}{OT2}{wncyss}{m}{n}
\DeclareMathSymbol{\sh}{\mathbin}{cyss}{`x}

\newtheorem{theorem}{Theorem}[section]
\newtheorem{lemma}[theorem]{Lemma}
\newtheorem{proposition}[theorem]{Proposition}
\newtheorem{corollary}[theorem]{Corollary}

\theoremstyle{definition}
\newtheorem{definition}[theorem]{Definition}
\newtheorem{example}[theorem]{Example}

\newtheorem{remark}[theorem]{Remark}

\DeclareFontEncoding{OT2}{}{}
\DeclareFontSubstitution{OT2}{cmr}{m}{n}
\DeclareSymbolFont{cyss}{OT2}{wncyss}{m}{n}
\DeclareMathSymbol{\sh}{\mathbin}{cyss}{`x}

\begin{document}

\title{Multi-poly-Bernoulli numbers and related zeta functions}

\author{Masanobu Kaneko\footnote{This work was supported by Japan Society for the Promotion of Science, 
Grant-in-Aid for Scientific Research (B) 23340010.}\, and Hirofumi Tsumura\footnote{This work was also supported by Japan Society for the Promotion of Science, 
Grant-in-Aid for Scientific Research (C) 15K04788.}}

\date{}

\maketitle

\begin{abstract}
  We construct and study a certain zeta function which interpolates multi-poly-Bernoulli numbers at non-positive integers
  and whose values at positive integers are linear combinations of multiple zeta values. 
  This function can be regarded as the one to be paired up with the $\xi$-function defined by Arakawa and the first-named author. 
  We show that both are closely related to the multiple zeta functions.  Further we define multi-indexed poly-Bernoulli numbers, and 
  generalize the duality formulas for poly-Bernoulli numbers by introducing more general zeta functions. 
\end{abstract}

\section{Introduction} \label{sec-1}

In this paper, we investigate the function defined by 
\begin{equation}
\eta(\kk;s)=\frac{1}{\Gamma(s)}\int_{0}^\infty t^{s-1}\frac{\Li_{\kk}(1-e^t)}{1-e^t}\,dt \label{etadef}
\end{equation}
and its generalizations, in connection with multi-poly-Bernoulli numbers and multiple zeta values
 (we shall give the precise definitions later in \S2).   
This function can be viewed as a twin sibling of the function $\xi(\kk;s)$,
\begin{equation}
\xi(\kk;s)=\frac{1}{\Gamma(s)}\int_0^\infty {t^{s-1}}\frac{\Li_{\kk}(1-e^{-t})}{e^t-1}\,dt,  \label{xidef}
\end{equation}
which was introduced and studied in \cite{AK1999}. 
The present paper may constitute a natural continuation of the work \cite{AK1999}. 

To explain our results in some detail, we first give an overview of the necessary background.
For an integer $k\in \mathbb{Z}$, two types of poly-Bernoulli numbers 
$\{\bb_n^{(k)}\}$ and $\{\cc_n^{(k)}\}$ are defined  as follows
(see Kaneko \cite{Kaneko1997} and Arakawa-Kaneko \cite{AK1999}, also Arakawa-Ibukiyama-Kaneko \cite{AIK2014}):
\begin{align}
&\frac{{\rm Li}_{k}(1-e^{-t})}{1-e^{-t}}=\sum_{n=0}^\infty \bb_n^{(k)}\frac{t^n}{n!},  \label{1-1}\\
&\frac{{\rm Li}_{k}(1-e^{-t})}{e^t-1}=\sum_{n=0}^\infty \cc_n^{(k)}\frac{t^n}{n!},  \label{1-2}
\end{align}
where ${\rm Li}_{k}(z)$ is the polylogarithm function defined by
\begin{equation}
{\rm Li}_{k}(z)=\sum_{m=1}^\infty \frac{z^m}{m^k}\quad (|z|<1). \label{1-3}
\end{equation}
Since ${\rm Li}_1(z)=-\log(1-z)$, we see that $\bb_n^{(1)}$ (resp. $C_n^{(1)}$) coincides with the ordinary Bernoulli number $B_n$ defined by
$$\frac{te^t}{e^t-1}=\sum_{n=0}^\infty B_n\frac{t^n}{n!}\qquad \biggl(\text{resp. }
\frac{t}{e^t-1}=\sum_{n=0}^\infty B_n\frac{t^n}{n!}\,\biggr)\,.$$
A number of formulas, including closed formulas of $\bb_n^{(k)}$  and $C_n^{(k)}$
in terms of the Stirling numbers of the second kind as well as the duality formulas
\begin{align}
& \bb_n^{(-k)}=\bb_{k}^{(-n)}, \label{1-4}\\
& \cc_n^{(-k-1)}=\cc_{k}^{(-n-1)} \label{1-5}
\end{align}
that hold for $k,n\in \mathbb{Z}_{\geq 0}$, have been established
(see \cite[Theorems\ 1\ and\ 2]{Kaneko1997} and \cite[\S\,2]{Kaneko-Mem}).
We also mention that Brewbaker \cite{Brew2008} gave a purely combinatorial interpretation of
the number $\bb_{n}^{(-k)}$ of negative upper index as the number of `Lonesum-matrices' with $n$ rows and $k$ columns.

A multiple version of $\bb_{n}^{(k)}$ is defined in \cite[p.\,202, Remarks\,(ii)]{AK1999} by
\begin{align}
&\frac{{\rm Li}_{k_1,\ldots,k_r}(1-e^{-t})}{(1-e^{-t})^r}=\sum_{n=0}^\infty \mathbb{B}_n^{(\kk)}\frac{t^n}{n!}\quad (\kk\in \mathbb{Z}),  \label{1-5-2}
\end{align}
where 
\begin{equation}
{\rm Li}_{k_1,\ldots,k_r}(z)=\sum_{1\leq m_1<\cdots<m_r}\frac{z^{m_r}}{m_1^{k_1}m_2^{k_2}\cdots m_r^{k_r}}  \label{1-7}
\end{equation}
is the multiple polylogarithm. 
Hamahata and Masubuchi \cite{HM2007-1,HM2007-2} investigated some properties of $\mathbb{B}_n^{(\kk)}$, 
and gave several generalizations of the known results in the single-index case. 
Based on this research, Bayad and Hamahata \cite{BH2012} further studied these numbers. 
Furusho \cite[p.\,269]{Fu2004} also refers to \eqref{1-5-2}. 

More recently, Imatomi, Takeda and the first-named author \cite{IKT2014} defined and studied another type 
of multi-poly-Bernoulli numbers given by
\begin{align}
&\frac{{\rm Li}_{k_1,\ldots,k_r}(1-e^{-t})}{1-e^{-t}}=\sum_{n=0}^\infty \bb_n^{(\kk)}\frac{t^n}{n!},  \label{1-6}\\
&\frac{{\rm Li}_{k_1,\ldots,k_r}(1-e^{-t})}{e^{t}-1}=\sum_{n=0}^\infty \cc_n^{(\kk)}\frac{t^n}{n!}  \label{1-6-2}
\end{align}
for $\kk\in \mathbb{Z}$. 
They proved several formulas for $\bb_n^{(\kk)}$ and $\cc_n^{(\kk)}$, 
and further gave an important relation between $\cc_{p-2}^{(\kk)}$ and the `finite multiple zeta value', that is, 
\begin{equation}
\sum_{1\leq m_1<\cdots <m_r<p}\frac{1}{m_1^{k_1}\cdots m_r^{k_r}} \equiv -C_{p-2}^{(k_1,\ldots,k_{r-1},k_r-1)} \mod p \label{1-8}
\end{equation}
for any prime number $p$.

The function \eqref{xidef} for $\kk \in \mathbb{Z}_{\geq 1}$ can be analytically continued to an entire function
of the complex variable $s\in\mathbb{C}$ (\cite[Sections $3$ and $4$]{AK1999}). 
The particular case $r=k=1$ gives 
$\xi(1;s)=s\zeta(s+1)$. Hence $\xi(\kk;s)$  can be regarded as a multi-indexed zeta function. It is shown in \cite{AK1999} that
the values at non-positive integers of $\xi(k;s)$ interpolate poly-Bernoulli numbers $\cc_m^{(k)}$, 
\begin{equation}
\xi(k;-m)=(-1)^m \cc_m^{(k)}  \label{1-10}
\end{equation}
for $k \in \mathbb{Z}_{\geq 1}$ and $m\in \mathbb{Z}_{\geq 0}$. And also by investigating $\xi(\kk;s)$ and its values at
positive integer arguments, one produces many relations among multiple zeta values defined by
\begin{equation}
\zeta(l_1,\ldots,l_r)=\sum_{1\leq m_1<\cdots<m_r}\frac{1}{m_1^{l_1}\cdots m_r^{l_r}}\ \left(\,=\,\Li_{l_1,\ldots,l_r}(1)\right)\label{MZV}
\end{equation}
for $l_1,\ldots,l_r\in \mathbb{Z}_{\geq 1}$ with $l_r\geq 2$ 
(\cite[Corollary 11]{AK1999}). 

Recently, further properties of $\xi(\kk;s)$ and related results have been given by several authors (see, for example, Bayad-Hamahata \cite{BH2011-1,BH2011-2}, Coppo-Candelpergher \cite{CC2010}, Sasaki \cite{Sasaki2012}, and
Young \cite{Yo}).

In this paper, we conduct a basic study of the function \eqref{etadef} and relate it to the multi-poly-Bernoulli numbers
$\bb_n^{(\kk)}$ as well as multiple zeta (or `zeta-star') values.  Note that the only difference in 
both definitions \eqref{etadef} and \eqref{xidef} is, up to sign, 
the arguments $1-e^t$ and $1-e^{-t}$ of $\Li_{\kk}(z)$ in the integrands.  
One sees in the main body of the paper a remarkable contrast between `$B$-type' poly-Bernoulli numbers and those 
of `$C$-type', and between multiple zeta and zeta-star values.  
We further investigate the case of non-positive indices $k_i$ in connection with
a yet more generalized `multi-indexed' poly-Bernoulli number.

The paper is organized as follows.  In \S2, we give the analytic continuation of $\eta(\kk;s)$ in the case of positive indices, and 
formulas for values at integer arguments (Theorems \ref{Th-Main-1} and \ref{T-3-7}).  In \S3, we study relations between two 
functions $\eta(\kk;s)$ and $\xi(\kk;s)$ (Proposition \ref{etaxi}), 
as well as relations with the single variable multiple zeta function (Definition \ref{mzfct} and Theorem \ref{xibyzeta}).
We turn in \S4 to the study of  $\eta(\kk;s)$ in the negative index case and give a certain duality formula for $B_m^{(\km)}$
(Definition \ref{Def-Main-2} and Theorems \ref{Th-Main-2} and \ref{Th-4-6}).
We carry forward the study of negative index case  in \S5 and define the `multi-indexed' poly-Bernoulli numbers  
$\{ \bb_{\mm}^{(\kk),(\aa)}\}$ for $(\kk) \in \mathbb{Z}^r$, $(\mm)\in \mathbb{Z}_{\geq 0}^r$ and $\aa\in \{1,\ldots,r\}$ 
(Definition \ref{Def-5-1}), 
which include \eqref{1-5-2} and \eqref{1-6} as special cases. We prove the `multi-indexed' duality formula for them in the case $\aa=r$ (Theorem \ref{T-5-8}).

\ 

\section{The function $\eta(\kk;s)$ for positive indices and its values at integers} \label{sec-2} 
\subsection{Analytic continuation and the values at non-positive integers}

We start with the definition in the case of positive indices.

\begin{definition}\label{Def-Main-1} 
For positive integers $k_1,\ldots,k_r\in \mathbb{Z}_{\geq 1}$, let
\begin{equation*}
\eta(\kk;s)=\frac{1}{\Gamma(s)}\int_{0}^\infty t^{s-1}\frac{\Li_{\kk}(1-e^t)}{1-e^t}dt \label{3-1}
\end{equation*}
for $s\in \mathbb{C}$ with ${\rm Re}(s)>1-r$, where $\Gamma(s)$ is the gamma function. When $r=1$, we often
denote $\eta(k;s)$ by $\eta_{k}(s)$. 
\end{definition}

The integral on the right-hand side converges absolutely in the domain ${\rm Re}(s)>1-r$, as
is seen from the following lemma.

\begin{lemma}\label{L-2-3} 
{\rm (i)}\ For $k_1,\ldots,k_r\in \mathbb{Z}_{\geq 1}$, the function $\Li_{\kk}(1-e^t)$ is holomorphic for 
$t\in \mathbb{C}$ with $|{\rm Im}(t)|<\pi$. 

\noindent
{\rm (ii)}\  For $k_1,\ldots,k_r\in \mathbb{Z}_{\geq 1}$ and $t\in \mathbb{R}_{>0}$, we have the estimates
\begin{equation}
\Li_{\kk}(1-e^t)=O\left(t^r\right)\qquad (t\to 0) \label{2-25}
\end{equation}
and
\begin{equation}
\Li_{\kk}(1-e^t)=O\left(t^{k_1+\cdots+k_r}\right)\qquad (t\to \infty). \label{2-3}
\end{equation}
\end{lemma}

\begin{proof}  As is well-known, we can regard the function $\Li_{\kk}(z)$ as a single-valued
holomorphic function in the simply connected domain $\mathbb{C}\smallsetminus [1,\infty)$, via the
process of iterated integration starting with $\Li_1(z)=\int_0^z dz/(1-z)$. 
Noting that $1-e^t \in [1,\infty)$ is equivalent to ${\rm Im}(t)= (2j+1)\pi$ for some $j\in \mathbb{Z}$,
we have the assertion {\rm (i)}.  

The estimate \eqref{2-25} is clear from the definition of $\Li_{\kk}(z)$, because its Taylor series at 
$z=0$ starts with the term $z^r/1^{k_1}\cdots r^{k_r}$.  As for \eqref{2-3}, we proceed by
induction on the `weight' $k_1+\cdots + k_r$ as
follows by using the formula 
\begin{equation}
\frac{d}{dz}\Li_{\kk}(z)=
\begin{cases}
\frac{1}{z}\Li_{k_1,\ldots,k_{r-1},k_r-1}(z) & \ (k_r>1)\\
\frac{1}{1-z}\Li_{k_1,\ldots,k_{r-1}}(z) & \ (k_r=1),
\end{cases}
\label{2-1}
\end{equation}
which is easy to derive and is the basis of the analytic continuation of $\Li_{\kk}(z)$ mentioned above. 
If $r=k_1=1$, then we have $\Li_1(1-e^t)=-t$ and the desired estimate holds. 
Suppose the weight $k$ is larger than $1$ and the assertion holds for any weight less than $k$.
 If $k_r>1$, then by \eqref{2-1} we have
\begin{align*}
|\Li_{\kk}(1-e^t)|&=\bigg|\int_{0}^{1-e^t}\frac{\Li_{k_1,\ldots,k_r-1}(u)}{u}du\bigg|\\
&=\bigg| \int_{0}^{t}\frac{1}{1-e^v}\Li_{k_1,\ldots,k_r-1}(1-e^v)(-e^v)dv\bigg|\qquad (u:=1-e^v)\\
& \leq \int_{0}^{\varepsilon}\bigg|{e^v}\frac{\Li_{k_1,\ldots,k_r-1}(1-e^v)}{e^v-1}\bigg|dv+\int_{\varepsilon}^{t}\bigg|\frac{e^v}{e^v-1}{\Li_{k_1,\ldots,k_r-1}(1-e^v)}\bigg|dv
\end{align*}
for small $\varepsilon>0$. The former integral is $O(1)$ because the integrand is continuous on $[0,\varepsilon]$. On the other hand, by induction hypothesis, the integrand of the latter integral is $O\left( v^{k_1+\cdots+k_r-1}\right)$ as $v\to\infty$. 
Therefore the latter integral is  $O\left( t^{k_1+\cdots+k_r}\right)$ as $t\to\infty$.  The case of $k_r=1$ is similarly proved also by using \eqref{2-1},
and is omitted here.
\end{proof}

We now show that the function $\eta(\kk;s)$ can be analytically continued to an entire function, 
and interpolates multi-poly-Bernoulli numbers $\bb_{m}^{(\kk)}$ at non-positive 
integer arguments.  

\begin{theorem}\label{Th-Main-1} 
For positive integers $k_1,\ldots,k_r\in \mathbb{Z}_{\geq 1}$, the function
$\eta(\kk;s)$ can be analytically continued to an entire function on the whole complex plane.
And the values of $\eta(\kk;s)$ at non-positive integers are given by
\begin{equation}
\eta(\kk;-m)=\bb_{m}^{(\kk)}\quad (m\in \mathbb{Z}_{\geq 0}). \label{3-2}
\end{equation}
In particular, $\eta_k(-m)=\bb_{m}^{(k)}$ for $k\in \mathbb{Z}_{\geq 1}$ and $m\in \mathbb{Z}_{\geq 0}$. 
\end{theorem}

\begin{proof}
In order to prove this theorem, we adopt here the method of contour integral representation (see, for example, \cite[Theorem 4.2]{Wa}). 
Let $\mathcal{C}$ be the standard contour, 
namely the path consisting of the positive real axis from the infinity to  (sufficiently small) $\varepsilon$ (`top side'), 
a counter clockwise circle $C_\varepsilon$ around the origin of radius $\varepsilon$, 
and the positive real axis from $\varepsilon$ to the infinity (`bottom side'). Let
\begin{align*}
H(\kk;s)&=\int_\mathcal{C}t^{s-1}\frac{\Li_{\kk}(1-e^t)}{1-e^t}dt\\
& =(e^{2\pi i s}-1)\int_{\varepsilon}^\infty t^{s-1}\frac{\Li_{\kk}(1-e^t)}{1-e^t}dt+\int_{C_\varepsilon}t^{s-1}
\frac{\Li_{\kk}(1-e^t)}{1-e^t}dt.
\end{align*}
It follows from Lemma \ref{L-2-3} that $H(\kk;s)$ is entire, 
because the integrand has no singularity on $\mathcal{C}$ and the contour integral is absolutely convergent for all $s\in \mathbb{C}$. 
Suppose ${\rm Re}(s)>1-r$. The last integral tends to $0$ as $\varepsilon \to 0$. Hence
$$\eta(\kk;s)=\frac{1}{(e^{2\pi i s}-1)\Gamma(s)}H(\kk;s),$$
which can be analytically continued to $\mathbb{C}$, and is entire. In fact $\eta(\kk;s)$ is holomorphic for ${\rm Re}(s)>0$, 
hence has no singularity at any positive integer. Set $s=-m\in \mathbb{Z}_{\leq 0}$. Then, by \eqref{1-6}, 
\begin{align*}
\eta(\kk;-m)& =\frac{(-1)^m m!}{2\pi i}H(\kk;-m)\\
& =\frac{(-1)^m m!}{2\pi i}\int_{C_\varepsilon}t^{-m-1}\sum_{n=0}^\infty \bb_n^{(\kk)}\frac{(-t)^n}{n!}dt=\bb_m^{(\kk)}.
\end{align*}
This completes the proof.
\end{proof}

\begin{remark}
Using the same method as above or the method used in \cite{AK1999}, we can establish the analytic continuation
of $\xi(\kk;s)$  to an entire function, and see that 
\begin{equation}
\xi(\kk;-m)=(-1)^m\cc_{m}^{(\kk)}\quad (m\in \mathbb{Z}_{\geq 0}) \label{3-3}
\end{equation}
for $\kk\in \mathbb{Z}_{\geq 1}$, which is a multiple version of \eqref{1-10}.
\end{remark}

\subsection{Values at positive integers}

About the values at positive integer arguments, we prove formulas for both $\xi(\kk;s)$ and
$\eta(\kk;s)$, for general index $(\kk)$.  These formulas generalize \cite[Theorem 9 (i)]{AK1999},
and  have remarkable similarity in that one obtains the formula for $\eta(\kk;s)$ just by replacing 
multiple zeta values in the one for $\xi(\kk;s)$ with multiple `zeta-star' values.
Recall the multiple zeta-star value is a real number defined by
\begin{equation}
\zeta^\star(l_1,\ldots,l_r)=\sum_{1\leq m_1\leq \cdots\leq m_r}\frac{1}{m_1^{l_1}\cdots m_r^{l_r}} \label{MZSV}
\end{equation}
for $l_1,\ldots,l_r\in \mathbb{Z}_{\geq 1}$ with $l_r\geq 2$. This was first studied (for general $r$)
by Hoffman in \cite{Ho}. 

To state our theorem, we further introduce some notation.  For an index set ${\bf k}=(\kk)\in\mathbb{Z}_{\ge1}^r$, 
put ${\bf k}_+=(k_1,\ldots,k_{r-1},k_r+1)$. 
The usual dual index of an {\it admissible}  index (i.e. the one that the last entry is greater than 
one) ${\bf k}$ is denoted by ${\bf k}^*$. For ${\bf j}=(j_1,\ldots,j_r)\in\mathbb{Z}_{\ge0}^r$, 
we write $\vert{\bf j}\vert=j_1+\cdots+j_r$ and call it the weight of ${\bf j}$, 
and $d({\bf j})=r$, the depth of ${\bf j}$. For two such indices ${\bf k}$ and ${\bf j}$ of the same depth, 
we denote by ${\bf k}+{\bf j}$ the index obtained by the component-wise addition,  
${\bf k}+{\bf j}=(k_1+j_1, \ldots, k_r+j_r)$, and by $b({\bf k};{\bf j})$ the quantity given by
\[   b({\bf k};{\bf j}):=\prod_{i=1}^r \binom{k_i+j_i-1}{j_i}. \]

\begin{theorem}\label{T-3-7}\ \ For any index set ${\bf k}=(\kk)\in\mathbb{Z}_{\ge1}^r$
and any $m\in \mathbb{Z}_{\geq 1}$, we have 
\begin{equation}\label{xivalue} 
\xi(\kk;m)=\sum_{\vert{\bf j}\vert=m-1,\,d({\bf j})=n} b(({\bf k}_+)^*;{\bf j})\,
\zeta(({\bf k}_+)^*+{\bf j})
\end{equation}
and
\begin{equation}\label{etavalue}
 \eta(\kk;m)=(-1)^{r-1}\sum_{\vert{\bf j}\vert=m-1,\, d({\bf j})=n} b(({\bf k}_+)^*;{\bf j})\,
\zeta^\star(({\bf k}_+)^*+{\bf j}),
\end{equation}
where both sums are over all ${\bf j}\in\mathbb{Z}_{\ge0}^r$ of weight $m-1$ and depth 
$n:=d({\bf k}_+^*) \ (=\vert{\bf k}\vert+1-d({\bf k}))$.

In particular, we have
\[ \qquad \xi(\kk;1)=\zeta(({\bf k}_+)^*)\quad(=\zeta({\bf k}_+),\ \text{by the duality of multiple zeta values}) \]
and
\[ \eta(\kk;1)=(-1)^{r-1}\zeta^\star(({\bf k}_+)^*).\]
\end{theorem}

In order to prove the theorem, we give certain multiple integral expressions of the functions 
$\xi(\kk;s)$ and $\eta(\kk;s)$.

\begin{proposition}  
Notations being as above, write $({\bf k}_+)^*=(l_1,\ldots,l_{n})$. Then we have, for $\re(s)>1-r$, 
\begin{enumerate}[{\rm (i)}]
\item \hfill
\vspace{-\abovedisplayskip}
\vspace{-\baselineskip}
\begin{align*}
\xi(\kk;s)&= \frac{1}{\prod_{i=1}^{n} \Gamma(l_i)\cdot\Gamma(s)}\int_0^\infty\cdots\int_0^\infty 
(x_1+\cdots+x_{n})^{s-1}x_{1}^{l_1-1}\cdots x_{n}^{l_{n}-1}\nonumber\\
& \qquad \times\frac1{e^{x_1+\cdots+x_{n}}-1}\cdot \frac1{e^{x_2+\cdots+x_{n}}-1}\cdots\cdots  
\frac1{e^{x_{n}}-1}
dx_1\cdots dx_{n}.
\end{align*}
\item \hfill
\vspace{-\abovedisplayskip}
\vspace{-\baselineskip}
\begin{align*}
\eta(\kk;s)& = \frac{(-1)^{r-1}}{\prod_{i=1}^{n} \Gamma(l_i)\cdot\Gamma(s)}\int_0^\infty\cdots\int_0^\infty 
(x_1+\cdots+x_{n})^{s-1}x_{1}^{l_1-1}\cdots x_{n}^{l_{n}-1}\nonumber\\
& \qquad \times\frac1{e^{x_1+\cdots+x_{n}}-1}\cdot \frac{e^{x_2+\cdots+x_{n}}}{e^{x_2+\cdots+x_{n}}-1}\cdots\cdots  
\frac{e^{x_{n}}}{e^{x_{n}}-1}
dx_1\cdots dx_{n}.
\end{align*}
\end{enumerate}
\end{proposition}

\begin{proof}  First write the index $(\kk)$ as 
\[  (\kk)=(\underbrace{1,\ldots,1}_{a_1-1},b_1+1,\ldots, \underbrace{1,\ldots,1}_{a_h-1},b_h+1),\]
with (uniquely determined) integers $h\ge1,\ a_i\ge1\ (1\le i\le h),\, b_i\ge1\ (1\le i\le h-1)$, and $b_h\ge0$.
Then, by performing the intermediate integrals of repeated $dz/(1-z)$ in the standard iterated integral 
coming from \eqref{2-1}, we obtain the following 
iterated integral expression of the multiple polylogarithm $\Li_{\kk}(z)$:
\begin{align}
\Li_{\kk}(z)&=\underbrace{\int_0^z\frac{dx_h}{x_h}\int_0^{x_h}
\cdots\cdots\int_0^{x_h}\frac{dx_h}{x_h}}_{b_h}\int_0^{x_h}\frac1{a_h!}\log\biggl(\frac{1-x_{h-1}}{1-x_{h}}\biggr)^{a_h}\,\frac{dx_{h-1}}{x_{h-1}}\nonumber\\
&\cdot\underbrace{\int_0^{x_{h-1}}\frac{dx_{h-1}}{x_{h-1}}\cdots\cdots \int_0^{x_{h-1}}
\frac{dx_{h-1}}{x_{h-1}}}_{b_{h-1}-1}\int_0^{x_{h-1}}\frac1{a_{h-1}!}
\log\biggl(\frac{1-x_{h-2}}{1-x_{h-1}}\biggr)^{a_{h-1}}\,\frac{dx_{h-2}}{x_{h-2}}\cdots\cdots\nonumber\\
&\cdots\cdots \underbrace{\int_0^{x_{3}}\frac{dx_{3}}{x_{3}}\cdots\cdots \int_0^{x_{3}}
\frac{dx_{3}}{x_{3}}}_{b_{3}-1}\int_0^{x_3}\frac1{a_3!}\log\biggl(\frac{1-x_2}{1-x_{3}}\biggr)^{a_3}
\frac{dx_2}{x_2}\underbrace{\int_0^{x_2}\frac{dx_2}{x_2}
\cdots\int_0^{x_2}\frac{dx_2}{x_2}}_{b_2-1}\nonumber\\
&\cdot\int_0^{x_2}\frac1{a_2!}\log\biggl(\frac{1-x_1}{1-x_{2}}\biggr)^{a_2}
\underbrace{\frac{dx_1}{x_1}\int_0^{x_1}\cdots\cdots\int_0^{x_1}\frac{dx_1}{x_1}\int_0^{x_1}}_{b_1-1}
\frac{\bigl(-\log(1-x)\bigr)^{a_1}}{a_1!}\frac{dx}{x}.\nonumber
\end{align}
Here, to ease notation, we used the same variable in the repetitions of integrals $\int_0^x dx/x$,  
and we understand $x_h=z$ if $b_h=0$. The paths of integrations are in the domain $\mathbb{C}\setminus [1,\infty)$,
and the formula is valid for $z\in\mathbb{C}\setminus [1,\infty)$.  We may check this formula by differentiating 
both sides repeatedly and using \eqref{2-1}.  Putting $z=1-e^{-t}$ and $1-e^t$, changing variables accordingly,
and suitably labeling the variables, we obtain 
\begin{align}\label{liiter1}
&\Li_{\kk}(1-e^{-t})=\int_0^t\int_0^{t_{b_1+\cdots+b_h}}\cdots\int_0^{t_2}
\underbrace{\frac1{e^{t_{b_1+\cdots+b_h}}-1}\ \cdots\cdots\  \frac1{e^{t_{b_1+\cdots+b_{h-1}+2}}-1}}_{b_h-1} \\
&\qquad\times\frac1{a_h!}\frac{(t_{b_1+\cdots+b_{h-1}+1}-t_{b_1+\cdots+b_{h-1}})^{a_h}}{e^{t_{b_1+\cdots+b_{h-1}+1}}-1}
\ \cdot\ \underbrace{\frac1{e^{t_{b_1+\cdots+b_{h-1}}}-1}\cdots\cdots
\frac1{e^{t_{b_1+\cdots+b_{h-2}+2}}-1}}_{b_{h-1}-1}\nonumber\\
&\qquad \times\cdots\cdots\nonumber\\
&\qquad\times\frac1{a_3!}\frac{(t_{b_1+b_2+1}-t_{b_1+b_2})^{a_3}}{e^{t_{b_1+b_2+1}}-1}
\ \cdot\ \underbrace{\frac1{e^{t_{b_1+b_2}}-1}\cdots\cdots 
\frac1{e^{t_{b_1+2}}-1}}_{b_2-1}\nonumber\\
&\qquad\times\frac1{a_2!}\frac{(t_{b_1+1}-t_{b_1})^{a_2}}{e^{t_{b_1+1}}-1}
\ \cdot\ \underbrace{\frac1{e^{t_{b_1}}-1}\cdots\cdots 
\frac1{e^{t_2}-1}}_{b_{1}-1}\cdot\frac1{a_1!}\frac{t_1^{a_1}}{e^{t_1}-1}\,dt_1\,dt_2\,\cdots\,dt_{b_1+\cdots+b_h},\nonumber
\end{align}
and
\begin{align}\label{liiter2}
&\Li_{\kk}(1-e^t)=(-1)^r\int_0^t\int_0^{t_{b_1+\cdots+b_h}}\cdots\int_0^{t_2}
\underbrace{\frac{e^{t_{b_1+\cdots+b_h}}}{e^{t_{b_1+\cdots+b_h}}-1}\ \cdots\cdots\  
\frac{e^{t_{b_1+\cdots+b_{h-1}+2}}}{e^{t_{b_1+\cdots+b_{h-1}+2}}-1}}_{b_h-1} \\
&\quad\times\frac1{a_h!}\frac{(t_{b_1+\cdots+b_{h-1}+1}-t_{b_1+\cdots+b_{h-1}})^{a_h}
e^{t_{b_1+\cdots+b_{h-1}+1}}}{e^{t_{b_1+\cdots+b_{h-1}+1}}-1}
\ \cdot\ \underbrace{\frac{e^{t_{b_1+\cdots+b_{h-1}}}}{e^{t_{b_1+\cdots+b_{h-1}}}-1}\cdots\cdots 
\frac{e^{t_{b_1+\cdots+b_{h-2}+2}}}{e^{t_{b_1+\cdots+b_{h-2}+2}}-1}}_{b_{h-1}-1}\nonumber\\
&\quad \times\cdots\cdots\nonumber\\
&\quad\times\frac1{a_3!}\frac{(t_{b_1+b_2+1}-t_{b_1+b_2})^{a_3}e^{t_{b_1+b_2+1}}}{e^{t_{b_1+b_2+1}}-1}
\ \cdot\ \underbrace{\frac{e^{t_{b_1+b_2}}}{e^{t_{b_1+b_2}}-1}\cdots\cdots 
\frac{e^{t_{b_1+2}}}{e^{t_{b_1+2}}-1}}_{b_2-1}\nonumber\\
&\quad\times\frac1{a_2!}\frac{(t_{b_1+1}-t_{b_1})^{a_2}e^{t_{b_1+1}}}{e^{t_{b_1+1}}-1}
\ \cdot\ \underbrace{\frac{e^{t_{b_1}}}{e^{t_{b_1}}-1}\cdots\cdots 
\frac{e^{t_2}}{e^{t_2}-1}}_{b_{1}-1}\cdot
\frac1{a_1!}\frac{t_1^{a_1}e^{t_1}}{e^{t_1}-1}\,dt_1\,dt_2\,\cdots\,dt_{b_1+\cdots+b_h}.\nonumber
\end{align}
The factor $(-1)^r$ on the right of \eqref{liiter2} comes from $(-1)^{a_1+\cdots+a_h}=(-1)^r$. 
Plugging \eqref{liiter1} and \eqref{liiter2} into the definitions \eqref{xidef} and \eqref{etadef} respectively  
and making the change of variables 
$$t=x_1+\cdots+x_n,\,t_{b_1+\cdots+b_h}=x_2+\cdots+x_n,
\,t_{b_1+\cdots+b_h-1}=x_3+\cdots+x_n,\ldots,t_2=x_{n-1}+x_n,\,t_1=x_n,$$ 
we obtain the proposition.  One should note that the dual index
$({\bf k}_+)^*=(l_1,\ldots,l_{n})$ is given by 
\[ ({\bf k}_+)^*=(\underbrace{1,\ldots,1}_{b_h},a_h+1,\underbrace{1,\ldots,1}_{b_{h-1}-1},a_{h-1}+1,
\ldots, \underbrace{1,\ldots,1}_{b_1-1},a_1+1) \]
and the depth $n$ is equal to $b_1+\cdots+b_h+1$, and that (the trivial) $x_i^{l_i-1}=1$ when $l_i=1$. 
\end{proof}

\begin{proof}[Proof of Theorem~\ref{T-3-7}]  Set $s=m$ in the integral expressions in the proposition, and expand
$(x_1+\cdots+x_k)^{m-1}$ by the multinomial theorem.  Then the formula in the theorem follows from 
the lemma below.  
\end{proof}
\begin{lemma}
For $l_1,\ldots,l_r\in \mathbb{Z}_{\geq 1}$ with $l_r\geq 2$,  we have
\[ \zeta(l_1,\ldots,l_r) \ = \frac{1}{\prod_{j=1}^{r}\Gamma(l_j)}\int_0^\infty\!\!\cdots\int_0^\infty 
\frac{x_1^{l_1-1}\cdots x_r^{l_r-1}}{e^{x_1+\cdots+x_r}-1}\cdot \frac1{e^{x_{2}+\cdots+x_{r}}-1}\cdots\cdots 
\frac1{e^{x_r}-1}dx_1\cdots dx_r \]
and
\[ \zeta^\star(l_1,\ldots,l_r) 
 = \frac{1}{\prod_{j=1}^{r}\Gamma(l_j)}\int_0^\infty\!\!\cdots\int_0^\infty 
\frac{x_1^{l_1-1}\cdots x_r^{l_r-1}}{e^{x_1+\cdots+x_r}-1}\cdot \frac{e^{x_{2}+\cdots+x_r}}{e^{x_{2}+\cdots+x_{r}}-1}
\cdots\cdots \frac{e^{x_r}}{e^{x_r}-1}dx_1\cdots dx_r.
\]
\end{lemma}

\begin{proof}
The first formula is given in \cite[Theorem 3 (i)]{AK1999}.  As for the second, we may proceed similarly 
by using $n^{-s}=\Gamma(s)^{-1}\int_0^\infty t^{s-1}e^{-nt}\,dt$ to have
\begin{align*}
&\zeta^\star(l_1,\ldots,l_r) =\sum_{m_1=1}^\infty\ \sum_{m_2,\ldots, m_r=0}^\infty
\frac1{m_1^{l_1}(m_1+m_2)^{l_2}\cdots\cdots(m_1+\cdots +m_r)^{l_r}}\\
&=\frac{1}{\prod_{j=1}^{r}\Gamma(l_j)}  \sum_{m_1=1}^\infty\ \sum_{m_2,\ldots, m_r=0}^\infty
\int_0^\infty\cdots\int_0^\infty x_1^{l_1-1} e^{-m_1x_1}\cdot x_2^{l_2-1}e^{-(m_1+m_2)x_2}\cdots\\
&\qquad\qquad\qquad\qquad\qquad\qquad\qquad\qquad\qquad\qquad
\cdots x_r^{l_r-1}e^{-(m_1+\cdots+m_r)x_r}\ dx_1\cdots dx_r\\
&=\frac{1}{\prod_{j=1}^{r}\Gamma(l_j)}  \sum_{m_1=1}^\infty\ \sum_{m_2,\ldots, m_r=0}^\infty
\int_0^\infty\cdots\int_0^\infty x_1^{l_1-1}\cdots x_r^{l_r-1}
e^{-m_1(x_1+\cdots+x_r)}\cdot e^{-m_2(x_2+\cdots+x_r)}\cdots\\
&\qquad\qquad\qquad\qquad\qquad\qquad\qquad\qquad\qquad\qquad
\cdots e^{-m_r x_r}\ dx_1\cdots dx_r\\
&= \frac{1}{\prod_{j=1}^{r}\Gamma(l_j)}\int_0^\infty\!\!\cdots\int_0^\infty 
\frac{x_1^{l_1-1}\cdots x_r^{l_r-1}}{e^{x_1+\cdots+x_r}-1}\cdot \frac{e^{x_{2}+\cdots+x_r}}{e^{x_{2}+\cdots+x_{r}}-1}
\cdots\cdots \frac{e^{x_r}}{e^{x_r}-1}dx_1\cdots dx_r.
\end{align*}
\end{proof}

We record here one corollary to the theorem in the case of $\eta_k(m)$ (compare with the similar formula 
in \cite[Theorem 9 (i)]{AK1999}). Noting $(k+1)^*=(\underbrace{1,\ldots,1}_{k-1},2)$,
we have

\begin{corollary}  For $k,m\ge1$, we have
\begin{equation}\label{3-8}
\eta_k(m)=\sum_{j_1,\ldots,j_{k-1}\ge1, j_k\ge2\atop j_1+\cdots+j_k=k+m}
(j_k-1)\zeta^\star(j_1,\ldots,j_{k-1},j_k). 
\end{equation}
\end{corollary}

\section{Relations among the functions $\xi,\,\eta$ and $\zeta$, and their consequences to 
multiple zeta values and multi-poly-Bernoulli numbers}\label{sec-3}

In this section, we first deduce that each of the functions $\eta$ and $\xi$ can be written
as a linear combination of the other by the {\it same} formula.  This is a consequence of the
so-called Landen-type connection formula for the multiple polylogarithm $\Li_{\kk}(z)$.
We then establish a formula for $\xi(\kk;s)$ in terms of the single-variable multiple zeta function
\begin{equation}
\zeta(l_1,\ldots,l_r;s)=\sum_{1\leq m_1<\cdots<m_r<m}\frac{1}{m_1^{l_1}\cdots m_r^{l_r}m^s}  \label{mzfct}
\end{equation}
defined for positive integers $l_1,\ldots,l_r$, the analytic continuation of which has been given in \cite{AK1999}
(the analytic continuation of a more general multi-variable multiple zeta function is established in \cite{AET2001}).
This answers to the question posed in \S5 of \cite{AK1999}. As a result, the function $\eta(\kk;s)$ can
also be written by the multiple zeta functions of the type above.
We then give a formula for values at positive integers of $\xi(\kk;s)$, and hence of $\eta(\kk;s)$, 
in terms of the `shuffle regularized values' of multiple zeta values, and thereby derive some consequences 
on the values of $\eta_k(s)$.

Let ${\bf k}=(k_1,\ldots,k_r)\in \mathbb{Z}_{\ge1}^r$ be an index set.
Recall that ${\bf k}$ is said to be admissible if the last entry $k_r$ is greater than 1,
the weight of  ${\bf k}$ is the sum $k_1+\cdots+k_r$, and 
the depth is the length $r$ of the index.  For two indices ${\bf k}$ and ${\bf k'}$ of
the same weight, we say ${\bf k'}$ refines ${\bf k}$, denoted ${\bf k}\preceq {\bf k'}$,
if ${\bf k}$ is obtained from ${\bf k'}$ by replacing some commas by $+$'s.  For example,
$(5)=(2+3)\preceq(2,3),\ (2,3)=(1+1,2+1)\preceq(1,1,2,1)$, etc.  The standard expression of a multiple
zeta-star value as a sum of multiple zeta values is written as
\[   \zeta^\star({\bf k})=\sum_{{\bf k'}\preceq{\bf k}\atop \text{admissible}}\zeta({\bf k'}), \]
where the sum on the right runs over the admissible indices ${\bf k'}$ such that ${\bf k}$ refines ${\bf k'}$.

The following formula is known as the Landen connection formula for the 
multiple polylogarithm  (\cite[Proposition 9]{OU}).

\begin{lemma}\label{landen} For any index ${\bf k}$ of depth $r$, we have
\begin{equation}  
\Li_{\bf k}\bigl(\frac{z}{z-1}\bigr)=(-1)^r\sum_{{\bf k}\preceq{\bf k'}}\Li_{\bf k'}(z).
\end{equation}
\end{lemma}
We can prove this by induction on weight and by using \eqref{2-1}, see \cite{OU}.\\

By using this and noting $z/(z-1)=1-e^t$ (resp. $1-e^{-t}$) if $z=1-e^{-t}$ (resp. $1-e^t$),
we immediately obtain the following proposition.

\begin{proposition}\label{etaxi} Let ${\bf k}$ be any index set and $r$ its depth.  We have the relations
\begin{equation} \eta({\bf k};s)= (-1)^{r-1}\sum_{{\bf k}\preceq{\bf k'}}\xi({\bf k'};s)\label{etabyxi}
\end{equation}
and
\begin{equation}\xi({\bf k};s)= (-1)^{r-1}\sum_{{\bf k}\preceq{\bf k'}}\eta({\bf k'};s).\label{xibyeta}
\end{equation}
\end{proposition}

\begin{corollary}\label{coretaxi} Let $k$ be a positive integer.  Then we have
\begin{equation}\label{etakxi}  
\eta_k(s)=\sum_{{\bf k}:\text{weight $k$}} \xi({\bf k};s)
\end{equation}
and
\begin{equation}
\label{xiketa} \xi_k(s)=\sum_{{\bf k}:\text{weight $k$}} \eta({\bf k};s),
\end{equation}
where the sums run over all indices of weight $k$. Here we have written $\xi_k(s)$ for $\xi(k;s)$. 
\end{corollary}

\begin{proof}  The index $(k)$ is of depth $1$ and all indices of weight $k$ (admissible or
non-admissible) refine $(k)$.
\end{proof}

We mention here that, also by taking ${\bf k}=(k)$ in Lemma~\ref{landen} and setting 
$z=1-e^t$ or $1-e^{-t}$, one immediately obtains a kind of sum formulas for multi-poly-Bernoulli
numbers as follows (compare with similar formulas in \cite[Theorem 3.1]{Ima}).

\begin{corollary} For $k\ge1$ and $m\ge0$,  we have
\begin{equation} \bb_m^{(k)}=(-1)^m\sum_{k_1+\cdots+k_r=k\atop k_i,r\ge1} C_m^{(\kk)}\end{equation}
and 
\begin{equation} C_m^{(k)}=(-1)^m\sum_{k_1+\cdots+k_r=k\atop k_i,r\ge1} B_m^{(\kk)}.\end{equation}
\end{corollary}

Next, we prove an Euler-type connection formula for the multiple polylogarithm.
If an index ${\bf k}$ is of weight $\vert{\bf k}\vert$, we also say the multiple zeta value 
$\zeta({\bf k})$ is of weight $\vert{\bf k}\vert$.

\begin{lemma}\label{eulcon} Let ${\bf k}$ be any index.  Then we have
\begin{equation}
\Li_{\bf k}(1-z)=\sum_{{\bf k'},\,j\ge0}c_{\bf k}({\bf k'};j)
\Li_{{\scriptsize{\underbrace{1,\ldots,1}_j}}}(1-z)\Li_{\bf k'}(z),\label{euler}
\end{equation}
where the sum on the right runs over indices ${\bf k'}$ and integers $j\ge0$ that satisfy
$\vert{\bf k'}\vert+j\le \vert{\bf k}\vert$, and $c_{\bf k}({\bf k'};j)$ is a $\mathbb{Q}$-linear combination
of multiple zeta values of weight $\vert{\bf k}\vert-\vert{\bf k'}\vert-j$. We understand
$\Li_{\emptyset}(z)=1$ and $\vert\emptyset\vert=0$ for the empty index $\emptyset$, and the constant $1$ is 
regarded as a multiple zeta value of weight $0$.
\end{lemma}

\begin{proof}  We proceed by induction on the weight of ${\bf k}$.  When ${\bf k}=(1)$, the trivial
identity $\Li_1(1-z)=\Li_1(1-z)$ is the one asserted.  Suppose the weight $\vert{\bf k}\vert$ of
${\bf k}$ is greater than $1$ and assume the statement holds for any index of weight less than 
$\vert{\bf k}\vert$.  For ${\bf k}=(k_1,\ldots,k_r)$, set ${\bf k}_-=(k_1,\ldots,k_{r-1},k_r-1)$ and
${\bf k}_+=(k_1,\ldots,k_{r-1},k_r+1)$.

First assume that ${\bf k}$ is admissible.  Then, by \eqref{2-1} and  induction hypothesis, 
we have 
\[ \frac{d}{dz}\Li_{\bf k}(1-z)=-\frac{\Li_{{\bf k}_-}(1-z)}{1-z}=
-\frac1{1-z}\sum_{{\bf l},\,j\ge0}c_{{\bf k}_-}({\bf l};j)\Li_{{\scriptsize{\underbrace{1,\ldots,1}_j}}}(1-z)\Li_{\bf l}(z),
\] the right-hand side being of a desired form.  Here, again by \eqref{2-1}, we see that 
\[ \frac1{1-z}\Li_{{\scriptsize{\underbrace{1,\ldots,1}_j}}}(1-z)\Li_{\bf l}(z)
=\frac{d}{dz}\left(\sum_{i=0}^j \Li_{{\scriptsize{\underbrace{1,\ldots,1}_{j-i}}}}(1-z)\Li_{{\bf l},1+i}(z)\right).
\]
We therefore conclude 
\[ \Li_{\bf k}(1-z)=-\sum_{{\bf l},\,j\ge0}c_{{\bf k}_-}({\bf l};j)
\sum_{i=0}^j \Li_{{\scriptsize{\underbrace{1,\ldots,1}_{j-i}}}}(1-z)\Li_{{\bf l},1+i}(z)+C
\] with some constant $C$.  Since $\lim_{z\to 0}\Li_{{\scriptsize{\underbrace{1,\ldots,1}_{j-i}}}}(1-z)\Li_{{\bf l},1+i}(z)
=0$, we find $C=\zeta({\bf k})$ by setting $z\to0$, and obtain the desired expression for $\Li_{\bf k}(1-z)$. 

When ${\bf k}$ is not necessarily admissible, write ${\bf k}=({\bf k}_0,\underbrace{1,\ldots,1}_q)$
with an admissible ${\bf k}_0$ and $q\ge0$. We prove the identity by induction on $q$. 
The case $q=0$ (${\bf k}$ is admissible) is already done.  Suppose $q\ge1$ and assume the 
claim is true for smaller $q$.  Then by assumption we have the expression
\[ \Li_{{\bf k}_0,\scriptsize{\underbrace{1,\ldots,1}_{q-1}}}(1-z)=
\sum_{{\bf m},\,j\ge0}c_{{\bf k}'}({\bf m};j)\Li_{{\scriptsize{\underbrace{1,\ldots,1}_j}}}(1-z)\Li_{\bf m}(z),
\] 
where we have put ${\bf k}'=({\bf k}_0,\underbrace{1,\ldots,1}_{q-1})$.
 We multiply $\Li_1(1-z)$ on both sides. Then, by the shuffle product, the left-hand side 
becomes the sum of the form 
\[ q\, \Li_{\bf k}(1-z)+\sum_{{\bf k}_0': \text{admissible}}\Li_{{\bf k}_0',\scriptsize{\underbrace{1,\ldots,1}_{q-1}}}(1-z),
\] and each term in the sum is written in the claimed form by induction hypothesis. 
On the other hand, the right-hand side becomes also of the form desired because 
\[ \Li_1(1-z)\Li_{{\scriptsize{\underbrace{1,\ldots,1}_j}}}(1-z)=(j+1)\Li_{{\scriptsize{\underbrace{1,\ldots,1}_{j+1}}}}(1-z).
\]  Hence $\Li_{\bf k}(1-z)$ is of the form as claimed.
\end{proof}

With the lemma, we are now able to establish the following (see \cite[\S 5, Problem (i)]{AK1999}).

\begin{theorem}\label{xibyzeta}  Let ${\bf k}$ be any index set.  The function $\xi({\bf k};s)$ can be written 
in terms of multiple zeta functions as
\[ \xi({\bf k};s)=\sum_{{\bf k'},\,j\ge0}c_{\bf k}({\bf k'};j)\binom{s+j-1}{j}\zeta({\bf k'};s+j). \]
Here, the sum is over indices ${\bf k'}$ and integers $j\ge0$ satisfying 
$\vert{\bf k'}\vert+j\le \vert{\bf k}\vert$, and $c_{\bf k}({\bf k'};j)$ is a $\mathbb{Q}$-linear combination
of multiple zeta values of weight $\vert{\bf k}\vert-\vert{\bf k'}\vert-j$.  The index ${\bf k'}$ may
be $\emptyset$ and for this we set $\zeta(\emptyset;s+j)=\zeta(s+j)$.
\end{theorem}

\begin{proof} By setting $z=e^{-t}$ in the lemma and using 
\begin{equation}\label{2-2}
\Li_{\scriptsize{\underbrace{1,\ldots,1}_j}}(z)=\frac{(-\log(1-z))^j}{j!},
\end{equation} 
we have 
\[ \Li_{\bf k}(1-e^{-t})=\sum_{{\bf k'},\,j\ge0}c_{\bf k}({\bf k'};j)\frac{t^j}{j!}\Li_{\bf k'}(e^{-t}).
\]
Substituting this into the definition \eqref{xidef} of $\xi({\bf k};s)$ and using the formula (\cite[Proposition 2, (i)]{AK1999})
\[ \zeta({\bf k};s)=\frac1{\Gamma(s)}\int_0^\infty\frac{t^{s-1}}{e^t-1}\Li_{\bf k}(e^{-t})\,dt, \]
we immediately obtain the theorem.
\end{proof}

\begin{remark}  This theorem generalizes \cite[Theorem 8]{AK1999}, where the corresponding formula
for $\Li_{\scriptsize{\underbrace{1,\ldots,1}_{r-1},k} }(1-z)$ is
\begin{align*} 
\Li_{\scriptsize{\underbrace{1,\ldots,1}_{r-1},k} }(1-z) &= (-1)^{k-1} \sum_{j_1+\cdots+j_k=r+k\atop
\forall j_i\ge1} \Li_{\scriptsize{\underbrace{1,\ldots,1}_{j_k-1}} }(1-z)\Li_{j_1,\ldots,j_{k-1}}(z)\\
&+\sum_{j=0}^{k-2} (-1)^j \zeta(\underbrace{1,\ldots,1}_{r-1},k-j)\Li_{\scriptsize{\underbrace{1,\ldots,1}_{j}}}(z).
\end{align*}
As pointed out by Shu~Oi, one can deduce Lemma~\ref{eulcon} by induction using \cite[Prop.~5]{OU2}.
However,  to describe the right-hand side of the lemma explicitly is a different problem and neither proof
gives such a formula in general. See also \cite{Oi} for a related topic.  
\end{remark}

\begin{example}  Apart from the trivial case $(1,\ldots,1)$, examples of the identity in Lemma \ref{eulcon}
up to weight $4$ are:
\begin{eqnarray*}
\Li_2(1-z)&=&-\Li_2(z)-\Li_{1}(1-z)\Li_1(z)+\zeta(2),\\
\Li_3(1-z)&=&\Li_{1,2}(z)+\Li_{2,1}(z)+\Li_1(1-z)\Li_{1,1}(z)-\zeta(2)\Li_1(z)+\zeta(3),\\
\Li_{1,2}(1-z)&=&-\Li_3(z)-\Li_1(1-z)\Li_2(z)-\Li_{1,1}(1-z)\Li_1(z)+\zeta(3),\\
\Li_{2,1}(1-z)&=&2\Li_3(z)+\Li_1(1-z)\Li_2(z)+\zeta(2)\Li_1(1-z)-2\zeta(3),\\
\Li_4(1-z)&=&-\Li_{1,1,2}(z)-\Li_{1,2,1}(z)-\Li_{2,1,1}(z)-\Li_1(1-z)\Li_{1,1,1}(z)\\
&&\quad +\zeta(2)\Li_{1,1}(z)-\zeta(3)\Li_1(z)+\zeta(4),\\
\Li_{1,3}(1-z)&=&\Li_{1,3}(z)+\Li_{2,2}(z)+\Li_{3,1}(z)+\Li_1(1-z)\Li_{1,2}(z)+\Li_1(1-z)\Li_{2,1}(z)\\
&&\quad+\Li_{1,1}(1-z)\Li_{1,1}(z)-\zeta(3)\Li_1(z)+\frac14\zeta(4),\\
\Li_{2,2}(1-z)&=&-\Li_{2,2}(z)-2\Li_{3,1}(z)-\Li_1(1-z)\Li_{2,1}(z)-\zeta(2)\Li_1(1-z)\Li_1(z)\\
&&\quad-\zeta(2)\Li_2(z)+2\zeta(3)\Li_1(z)+\frac34\zeta(4),\\
\Li_{3,1}(1-z)&=&-2\Li_{1,3}(z)-\Li_{2,2}(z)-\Li_1(1-z)\Li_{1,2}(z)+\zeta(2)\Li_2(z)\\
&&\quad+\zeta(3)\Li_1(1-z)-\frac54\zeta(4),\\
\Li_{1,1,2}(1-z)&=&-\Li_4(z)-\Li_1(1-z)\Li_3(z)-\Li_{1,1}(1-z)\Li_2(z)\\
&&\quad -\Li_{1,1,1}(1-z)\Li_1(z)+\zeta(4),\\
\Li_{1,2,1}(1-z)&=&3\Li_4(z)+2\Li_1(1-z)\Li_3(z)+\Li_{1,1}(1-z)\Li_2(z)+\zeta(3)\Li_1(1-z)\\
&&\quad-3\zeta(4),\\
\Li_{2,1,1}(1-z)&=&-3\Li_4(z)-\Li_1(1-z)\Li_3(z)+\zeta(2)\Li_{1,1}(1-z)-2\zeta(3)\Li_1(1-z)\\
&&\quad+3\zeta(4).
\end{eqnarray*}

Accordingly, we have 
\begin{eqnarray*}
\xi(2;s)&=&-\zeta(2;s)-s\zeta(1;s+1)+\zeta(2)\zeta(s),\\
\xi(3;s)&=&\zeta(1,2;s)+\zeta(2,1;s)+s\zeta(1,1;s+1)-\zeta(2)\zeta(1;s)+\zeta(3)\zeta(s),\\
\xi(1,2;s)&=&-\zeta(3;s)-s\zeta(2;s+1)-\frac{s(s+1)}2\zeta(1;s+2)+\zeta(3)\zeta(s),\\
\xi(2,1;s)&=&2\zeta(3;s)+s\zeta(2;s+1)+\zeta(2)s\zeta(s+1)-2\zeta(3)\zeta(s),\\
\xi(4;s)&=&-\zeta(1,1,2;s)-\zeta(1,2,1;s)-\zeta(2,1,1;s)-s\zeta(1,1,1;s+1)\\
&&\quad +\zeta(2)\zeta(1,1;s)-\zeta(3)\zeta(1;s)+\zeta(4)\zeta(s),\\
\xi(1,3;s)&=&\zeta(1,3;s)+\zeta(2,2;s)+\zeta(3,1;s)+s\zeta(1,2;s+1)+s\zeta(2,1;s+1)\\
&&\quad+\frac{s(s+1)}2\zeta(1,1;s+2)-\zeta(3)\zeta(1;s)+\frac14\zeta(4)\zeta(s),\\
\xi(2,2;s)&=&-\zeta(2,2;s)-2\zeta(3,1;s)-s\zeta(2,1;s+1)-\zeta(2)s\zeta(1;s+1)\\
&&\quad-\zeta(2)\zeta(2;s)+2\zeta(3)\zeta(1;s)+\frac34\zeta(4)\zeta(s),\\
\xi(3,1;s)&=&-2\zeta(1,3;s)-\zeta(2,2;s)-s\zeta(1,2;s+1)+\zeta(2)\zeta(2;s)\\
&&\quad+\zeta(3)s\zeta(s+1)-\frac54\zeta(4)\zeta(s),\\
\xi(1,1,2;s)&=&-\zeta(4;s)-s\zeta(3;s+1)-\frac{s(s+1)}2\zeta(2;s+2)\\
&&\quad -\frac{s(s+1)(s+2)}6\zeta(1;s+3)+\zeta(4)\zeta(s),\\
\xi(1,2,1;s)&=&3\zeta(4;s)+2s\zeta(3;s+1)+\frac{s(s+1)}2\zeta(2;s+2)+\zeta(3)s\zeta(s+1)\\
&&\quad-3\zeta(4)\zeta(s),\\
\xi(2,1,1;s)&=&-3\zeta(4;s)-s\zeta(3;s+1)+\zeta(2)\frac{s(s+1)}2\zeta(s+2)-2\zeta(3)s\zeta(s+1)\\
&&\quad+3\zeta(4)\zeta(s).
\end{eqnarray*}
From these and \eqref{etakxi} of Corollary~\ref{coretaxi}, we have for instance
\begin{eqnarray*}
\eta_2(s)&=&\xi(2;s)+\xi(1,1;s)\\
&=&-\zeta(2;s)-s\zeta(1;s+1)+\zeta(2)\zeta(s)+\frac{s(s+1)}2\zeta(s+2),\\
\eta_3(s)&=&\xi(3;s)+\xi(1,2;s)+\xi(2,1;s)+\xi(1,1,1;s)\\
&=& \zeta(3;s)+\zeta(1,2;s)+\zeta(2,1;s)+s\zeta(1,1;s+1)-\frac{s(s+1)}2\zeta(1;s+2)\\
&&\quad-\zeta(2)\zeta(1;s)+\zeta(2)s\zeta(s+1)+\frac{s(s+1)(s+2)}6\zeta(s+3),\\
\eta_4(s)&=&\xi(4;s)+\xi(1,3;s)+\xi(2,2;s)+\xi(3,1;s)+\xi(1,1,2;s)+\xi(1,2,1;s)\\
&&\quad+\xi(2,1,1;s)+\xi(1,1,1,1;s)\\
&=& -\zeta(4;s)-\zeta(1,3;s)-\zeta(2,2;s)-\zeta(3,1;s)-\zeta(1,1,2;s)-\zeta(1,2,1;s)\\
&&\quad-\zeta(2,1,1;s)-s\zeta(1,1,1;s+1)+\zeta(2)\zeta(1,1;s)+\frac{s(s+1)}2\zeta(1,1;s+2)\\
&&\quad-\zeta(2)s\zeta(1;s+1)+\zeta(2)\frac{s(s+1)}2\zeta(s+2)
-\frac{s(s+1)(s+2)}6\zeta(1;s+3)\\
&&\quad+\frac74\zeta(4)\zeta(s)+\frac{s(s+1)(s+2)(s+3)}{24}\zeta(s+4).
 \end{eqnarray*}
\end{example}

Before closing this section, we present a curious observation.  Recall the formula
\[ \xi_k(m)=\zeta^\star(\underbrace{1,\ldots,1}_{m-1},k+1) \]
discovered by Ohno \cite{Ohno}.  Comparing this with the two formulas \eqref{3-8} and \cite[Corollary 10]{AK1999}, one may
expect $$\eta_k(m)\overset{?}{=}\zeta(\underbrace{1,\ldots,1}_{m-1},k+1).$$
This is not true in fact.  However, we found experimentally the identities
\begin{equation} \eta_k(m)=\eta_m(k) \label{etadual} \end{equation}
and 
\begin{equation} \sum_{j=1}^{k-1} (-1)^{j-1}\eta_{k-j}(j)=\begin{cases} 2(1-2^{1-k})\zeta(k)&\quad \text{($k$: even),}\\
0 &\quad \text{($k$: odd).} \end{cases} \label{etaLM} \end{equation}
These are respectively analogous to the duality relation  
\[\zeta(\underbrace{1,\ldots,1}_{m-1},k+1)
=\zeta(\underbrace{1,\ldots,1}_{k-1},m+1)\] and the relation
\[ \sum_{j=1}^{k-1} (-1)^{j-1}\zeta(\underbrace{1,\ldots,1}_{j-1},k-j+1)
=\begin{cases} 2(1-2^{1-k})\zeta(k)&\quad \text{($k$: even),}\\
0 &\quad \text{($k$: odd),} \end{cases} \]
which is a special case of the Le-Murakami relation \cite{LM} (or one can derive this from
the well-known generating series identity \cite{Ao}, \cite{Dr}
\[ 1-\sum_{k>j\geq1}\zeta(\underbrace{1,\dots,1}_{j-1},k-j+1)X^{k-j}\,Y^j
=\frac{\Gamma(1-X)\Gamma(1-Y)}{\Gamma(1-X-Y)}
\]
by setting $Y=-X$ and using the reflection formula for the gamma function.)

We are still not able to prove \eqref{etadual}\footnote{Quite recently, Shuji~Yamamoto communicated to the 
authors that he found a proof.}, but could prove \eqref{etaLM} by using the 
following general formula for the value $\xi({\bf k};m)$ and the relation \eqref{etabyxi} in Proposition~\ref{etaxi}.
For other aspects of `height one' multiple zeta values, see \cite{KS}.

\begin{proposition} Let ${\bf k}$ be any index and $m\ge1$ an integer.  Then we have
\begin{equation} \xi({\bf k};m)=(-1)^{m-1}\zeta^\sa({\bf k}_+,\underbrace{1,\ldots,1}_{m-1}), \label{xireg}
\end{equation}
where $\zeta^\sa$ stands for the `shuffle regularized' value, which is the constant term
of the shuffle regularized polynomial defined in \cite{IKZ}.
\end{proposition}

\begin{proof}  By making the change of variable $x=1-e^{-t}$ in the definition \eqref{xidef}, we have
\[ \xi({\bf k};s) = \frac1{\Gamma(s)}\int_0^1\bigl(-\log(1-x)\bigr)^{s-1}\Li_{\bf k}(x)\,\frac{dx}{x}. \]
Put $s=m$ and use \eqref{2-2} to obtain 
\[ \xi({\bf k};m) = \int_0^1\Li_{\scriptsize{\underbrace{1,\ldots,1}_{m-1}}}(x)\Li_{\bf k}(x)\,\frac{dx}{x}. \]
The regularization formula \cite[Eq. (5.2)]{IKZ}, together with the shuffle product of 
$\Li_{\scriptsize{\underbrace{1,\ldots,1}_{m-1}}}(x)\Li_{\bf k}(x)$, immediately gives \eqref{xireg}.
\end{proof}

By using \eqref{xireg} and \eqref{etakxi}, we can write $\eta_k(m)$ in terms of shuffle regularized values.
The following expression seems to follow from that formula by taking the dual, but we have not yet worked it out
in detail.
\[ \eta_k(m)\overset{?}{=}\binom{m+k}{k}\zeta(m+k)-\!\!\!\sum_{2\le r\le k+1\atop 
j_1+\cdots+j_r=m+k-r-1}\binom{j_1+\cdots+j_{r-1}}{k-r+1}\zeta(j_1+1,\cdots,j_{r-1}+1,j_r+2).
\]

\ 

\section{The function $\eta(\kk;s)$ for non-positive indices}\label{sec-4}

In this section, as in the case of positive indices, we construct $\eta$-functions with non-positive indices. 
It is known that $\Li_{-k}(z)$ can be expressed as 
$$\Li_{-k}(z)=\frac{P(z;k)}{(1-z)^{k+1}}$$
for $k\in \mathbb{Z}_{\geq 0}$, where $P(x;k)\in \mathbb{Z}[x]$ is a monic polynomial satisfying
\begin{align*}
& {\rm deg}\,P(x;k) =
\begin{cases}
1 & (k=0)\\
k & (k\geq 1),
\end{cases}
\\
& x \mid P(x;k)
\end{align*}
(see, for example, Shimura \cite[Equations (2.17),\,(4.2) and (4.6)]{Shimura};\ Note that the above $P(x;k)$ 
coincides with $xP_{k+1}(x)$ in \cite{Shimura}). 
We first extend this fact to multiple polylogarithms with non-positive indices as follows. 

\begin{lemma}\label{L-4-1} \ For $k_1,\ldots,k_r\in \mathbb{Z}_{\geq 0}$, there exists a polynomial 
$P(x;\kk)\in \mathbb{Z}[x]$ such that
\begin{align}
& \Li_{\km}(z)=\frac{P(z;\kk)}{(1-z)^{k_1+\cdots+k_r+r}}, \label{4-1}\\
& {\rm deg}\,P(x;\kk)=
\begin{cases}
r & (k_1=\cdots=k_r=0)\\
k_1+\cdots+k_r+r-1 & (\text{\rm otherwise}),
\end{cases}
\label{4-2}\\
& x^r\mid P(x;\kk).\label{4-3}
\end{align}
More explicitly, $P(x;\underbrace{0,0,\ldots,0}_{r})=x^r$.
\end{lemma}

\begin{proof}
We prove this lemma by the double induction on $r\geq 1$ and $K=k_1+\cdots+k_r\geq 0$. 
The case $r=1$ is as mentioned above. For $r\geq 2$, we assume the case of $r-1$ holds and consider the case of $r$. 
When $K=k_1+\cdots+k_r=0$, namely $k_1=\cdots=k_r=0$, we have
\begin{align*}
\Li_{0,\ldots,0}(z)&=\sum_{m_1<\cdots<m_r}x^{m_r}=\sum_{m_1<\cdots<m_{r-1}}\sum_{m_r=m_{r-1}+1}^\infty z^{m_r}\\
&=\frac{z}{1-z}\sum_{m_1<\cdots<m_{r-1}}z^{m_{r-1}}=\cdots=\frac{z^r}{(1-z)^r},
\end{align*}
which implies \eqref{4-1}--\eqref{4-3} hold, and also $P(x;{0,\ldots,0})=x^r$. Hence we assume the case $K=k_1+\cdots+k_r-1$ holds and consider the case $K=k_1+\cdots+k_r(\geq 1)$. We consider the two cases  $k_r=0$ and $k_r\geq 1$ separately. 
First we assume $k_r=0$. Then, by induction hypothesis, we have 
\begin{align*}
\Li_{-k_1,\ldots,-k_{r-1},0}(z)&=\sum_{m_1<\cdots<m_{r-1}}m_1^{k_1}\cdots m_{r-1}^{k_{r-1}}\sum_{m_r=m_{r-1}+1}^\infty z^{m_r}\\
& =\frac{z}{1-z}\sum_{m_1<\cdots<m_{r-1}}m_1^{k_1}\cdots m_{r-1}^{k_{r-1}}z^{m_{r-1}}\\
& =\frac{z}{1-z}\,\frac{P(z:k_1,\ldots,k_{r-1})}{(1-z)^{k_1+\cdots+k_{r-1}+r-1}}.
\end{align*}
Let $P(z;k_1,\ldots,k_{r-1},0)=zP(z;k_1,\ldots,k_{r-1})$. 
Then \eqref{4-1}--\eqref{4-3} hold.

Next we assume $k_r\geq 1$. Then, using the same formula as in \eqref{2-1} and the induction hypothesis, we have 
\begin{align*}
&\Li_{-k_1,\ldots,-k_{r-1},-k_r}(z)=z\frac{d}{dz}\Li_{-k_1,\ldots,-k_r+1}(z)\\
&=z\frac{d}{dz}\left(\frac{P(z:k_1,\ldots,k_{r}-1)}{(1-z)^{k_1+\cdots+k_{r}-1+r}}\right)\\
&=\frac{z\left\{P'(z:k_1,\ldots,k_{r}-1)(1-z)+(k_1+\cdots+k_r-1+r)P(z;k_1,\ldots,k_r-1)\right\}}{(1-z)^{k_1+\cdots+k_{r}+r}}.
\end{align*}
If $k_1=\cdots=k_{r-1}=0$ and $k_r=1$, then the numerator, that is, $P(0,\ldots,0,-1)$ equals  $rz^r$, using the above results. If not, 
the degree of the numerator equals $k_1+\cdots+k_r+r-1$ by induction hypothesis. 
The both cases satisfy \eqref{4-1}--\eqref{4-3}. This completes the proof of the lemma.
\end{proof}

\begin{remark}
In the case $r\geq 2$, $P(x;\kk)$ is not necessarily a monic polynomial. For example, we have $\Li_{0,-1}(z)=2z^2/(1-z)^3$, 
so $P(x;0,1)=2x^2$.
\end{remark}

We obtain from \eqref{4-1} and \eqref{4-2} that
\begin{equation}
\Li_{\km}(1-e^t)=\frac{P(1-e^t;\kk)}{e^{(k_1+\cdots+k_r+r)t}} =
\begin{cases}
O(1) & (k_1=\cdots=k_r=0)\\
O(e^{-t}) & (\text{\rm otherwise})
\end{cases}
\label{4-4}
\end{equation}
as $t\to \infty$, and from \eqref{4-3} that 
\begin{equation}
\Li_{\km}(1-e^t)=O(t^{r}) \quad (t\to 0).\label{4-4-2}
\end{equation}
Therefore we can define the following.

\begin{definition}\label{Def-Main-2} 
For $k_1,\ldots,k_r\in \mathbb{Z}_{\geq 0}$, define 
\begin{equation}
\eta(\km;s)=\frac{1}{\Gamma(s)}\int_{0}^\infty t^{s-1}\frac{\Li_{\km}(1-e^t)}{1-e^t}dt \label{4-5}
\end{equation}
for $s\in \mathbb{C}$ with ${\rm Re}(s)>1-r$. In the case $r=1$, denote $\eta(-k;s)$ by $\eta_{-k}(s)$.
\end{definition}

We see that the integral on the right-hand side of \eqref{4-5} is absolutely convergent for ${\rm Re}(s)>1-r$. Hence 
$\eta(\km;s)$ is holomorphic for ${\rm Re}(s)>1-r$. By the same method as in the proof of Theorem \ref{Th-Main-1} 
for $\eta(\kk;s)$, we can similarly obtain the following.

\begin{theorem}\label{Th-Main-2} 
For $k_1,\ldots,k_r\in \mathbb{Z}_{\geq 0}$, 
$\eta(\km;s)$ can be analytically continued to an entire function on the whole complex plane, and satisfies 
\begin{equation}
\eta(\km;-m)=\bb_{m}^{(\km)}\quad (m\in \mathbb{Z}_{\geq 0}). \label{4-6}
\end{equation}
In particular, $\eta_{-k}(-m)=\bb_{m}^{(-k)}$ $(k\in \mathbb{Z}_{\geq 0},\ m\in \mathbb{Z}_{\geq 0})$. 
\end{theorem}

It should be noted that $\xi(\km;s)$ cannot be defined by replacing $\{ k_j\}$ by $\{ -k_j\}$ in \eqref{xidef}. 
In fact, even if $r=1$ and $k=0$ in \eqref{xidef}, 
we see that 
$$\xi_0(s)=\frac{1}{\Gamma(s)}\int_{0}^\infty t^{s-1}\frac{\Li_0(1-e^{-t})}{e^t-1}dt=
\frac{1}{\Gamma(s)}\int_{0}^\infty t^{s-1}dt,$$
which is not convergent for any $s\in \mathbb{C}$. Therefore we modify the definition \eqref{xidef} 
as follows.

\begin{definition}
For $k_1,\ldots,k_r\in \mathbb{Z}_{\geq 0}$ with $(k_1,\ldots,k_r)\neq (0,\ldots,0)$, define 
\begin{equation}
\widetilde{\xi}(\km;s)=\frac{1}{\Gamma(s)}\int_{0}^\infty t^{s-1}\frac{\Li_{\km}(1-e^t)}{e^{-t}-1}dt \label{4-5-2}
\end{equation}
for $s\in \mathbb{C}$ with ${\rm Re}(s)>1-r$. In the case $r=1$, denote $\widetilde{\xi}(-k;s)$ by 
$\widetilde{\xi}_{-k}(s)$ for $k\geq 1$.
\end{definition}

We see from \eqref{4-4} and \eqref{4-4-2} that \eqref{4-5-2} is well-defined. 
Also it is noted that $\widetilde{\xi}(\kk;s)$ cannot be defined by replacing $\{-k_j\}$ by $\{k_j\}$ in \eqref{4-5-2} for $(k_j)\in \mathbb{Z}_{\geq 1}^r$.

In a way parallel to deriving Theorem \ref{Th-Main-2}, we can obtain the following.

\begin{theorem}\label{Th-Main-2-2} 
For $k_1,\ldots,k_r\in \mathbb{Z}_{\geq 0}$ with $(k_1,\ldots,k_r)\neq (0,\ldots,0)$, 
$\widetilde{\xi}(\km;s)$ can be analytically continued to an entire function on the whole complex plane, and satisfies 
\begin{equation}
\widetilde{\xi}(\km;-m)=\cc_{m}^{(\km)}\quad (m\in \mathbb{Z}_{\geq 0}). \label{4-6-2}
\end{equation}
In particular, $\widetilde{\xi}_{-k}(-m)=\cc_{m}^{(-k)}$ $(k\in \mathbb{Z}_{\geq 1},\ m\in \mathbb{Z}_{\geq 0})$. 
\end{theorem}

Next we give certain duality formulas for $\bb_n^{(\kk)}$ which is a generalization of \eqref{1-4}. 
To state this, we define another type of multi-poly-Bernoulli numbers by
\begin{equation}
\begin{split}
& \sum_{a=0}^{r-1}(-1)^a\binom{r-1}{a} \sum_{l_1,\ldots,l_r\geq 1}\frac{\prod_{j=1}^{r}\left(1-e^{-\sum_{\nu=j}^{r}x_\nu}\right)^{l_j-1}}{(l_1+\cdots+l_r-a)^s}\\
& \quad =\sum_{m_1,\ldots,m_r\geq 0}\ba_{m_1,\ldots,m_r}^{(s)}\frac{x_1^{m_1}\cdots x_r^{m_r}}{m_1! \cdots m_r!}
\end{split}
 \label{4-7}
\end{equation}
for $s\in \mathbb{C}$. In the case $r=1$, we see that $\ba_{m}^{(k)}=B_{m}^{(k)}$ for $k\in \mathbb{Z}$. 
Then we obtain the following result which is a kind of the duality formula. 
In fact, this coincides with \eqref{1-4} in the case $r=1$.

\begin{theorem}\label{Th-4-6}\ \ For $\kk \in \mathbb{Z}_{\geq 0}$, 
\begin{equation}
\eta(\km;s)=\ba_{\kk}^{(s)}.   \label{4-8}
\end{equation}
Therefore, for $m\in \mathbb{Z}_{\geq 0}$,
\begin{equation}
B_m^{(\km)}=\ba_{\kk}^{(-m)}.  \label{4-9}
\end{equation}
\end{theorem}

\begin{proof}
We first prepare the following relation which will be proved in the next section (see Lemma \ref{L-5-5}):
\begin{equation}
\prod_{j=1}^r \frac{e^{\sum_{\nu=j}^{r}x_\nu}(1-e^t)}{1-e^{\sum_{\nu=j}^{r}x_\nu}(1-e^t)}=\sum_{\kk \geq 0}\Li_{\km}(1-e^t)\frac{x_1^{k_1}\cdots x_r^{k_r}}{k_1!\cdots k_r!} \label{4-10}
\end{equation}
holds around the origin. Let
$$\mathcal{F}(x_1,\ldots,x_r;s)=\sum_{\kk \geq 0}\eta(\km;s)\frac{x_1^{k_1}\cdots x_r^{k_r}}{k_1!\cdots k_r!}.$$
As a generalization of \cite[Proposition 5]{IKT2014}, we have from \eqref{4-10} that
\begin{align*}
 \mathcal{F}(x_1,\ldots,x_r;s)& =\frac{1}{\Gamma(s)}\int_0^\infty \frac{t^{s-1}}{1-e^t}\prod_{j=1}^r \frac{e^{\sum_{\nu=j}^{r}x_\nu}(1-e^t)}{1-e^{\sum_{\nu=j}^{r}x_\nu}(1-e^t)}dt\\
& =\frac{1}{\Gamma(s)}\int_0^\infty {t^{s-1}}{(1-e^t)^{r-1}e^{-rt}}\prod_{j=1}^r \frac{1}{1-e^{-t}
\left(1-e^{-\sum_{\nu=j}^{r}x_\nu}\right)}dt\\
& =\frac{1}{\Gamma(s)}\sum_{a=0}^{r-1}(-1)^a\binom{r-1}{a}\sum_{m_1,\ldots,m_r\geq 0}\prod_{j=1}^r \left(1-e^{-\sum_{\nu=j}^r x_\nu}\right)^{m_j}\\
& \qquad \times \int_0^\infty {t^{s-1}}{e^{(a-r)t}}\prod_{j=1}^r e^{-m_jt}\ dt\\& =\sum_{a=0}^{r-1}
(-1)^a\binom{r-1}{a}\sum_{m_1,\ldots,m_r\geq 0}\frac{\prod_{j=1}^r \left(1-e^{-\sum_{\nu=j}^r x_\nu}\right)^{m_j}}{(m_1+\cdots+m_r+r-a)^s}.
\end{align*}
Therefore, by \eqref{4-7}, we obtain \eqref{4-8}. Further, setting $s=-m$ in \eqref{4-8} and using \eqref{4-6}, we obtain \eqref{4-9}.
\end{proof}

\begin{remark}
In the case $r=1$, \eqref{4-8} implies $\eta_{-k}(s)=\bb_{k}^{(s)}$. Thus, using Theorem \ref{Th-Main-2}, we obtain the duality formula \eqref{1-4}, which is also written as 
\begin{equation}
\eta_{-k}(-m)=\eta_{-m}(-k) 
\end{equation}
for $k,m\in \mathbb{Z}_{\geq 0}$. 
This is exactly contrasted with the positive index case \eqref{etadual}. 
Furthermore, by the same method, we can show that $\widetilde{\xi}_{-k-1}(-m)=C_{k}^{(-m-1)}$ for $k,m \in \mathbb{Z}_{\geq 0}$. Hence, using Theorem \ref{Th-Main-2-2} in the case $r=1$, we obtain the duality formula \eqref{1-5}.
\end{remark}

\begin{example}
When $r=2$, we can calculate directly from \eqref{4-7} that 
$\ba_{1,0}^{(s)}=3^{-s}-2^{-s}.$ 
On the other hand, as mentioned in Lemma \ref{L-4-1}, we have
$\Li_{-1,0}(z)=z^2/(1-z)^3$. Hence the left-hand side of \eqref{1-6} equals 
$$\frac{{\rm Li}_{-1,0}(1-e^{-t})}{1-e^{-t}}=\frac{1-e^{-t}}{e^{-3t}}=e^{3t}-e^{2t},$$
hence $\bb_m^{(-1,0)}=3^m-2^m$. Thus we can verify $\bb_m^{(-1,0)}=\ba_{1,0}^{(-m)}$.
\end{example}

\ 

\section{Multi-indexed poly-Bernoulli numbers and duality formulas}\label{sec-5}

In this section, we define multi-indexed poly-Bernoulli numbers (see Definition \ref{Def-5-1}) 
and prove the duality formula for them, namely a multi-indexed version of \eqref{1-4} (see Theorem \ref{T-5-8}). 

For this aim, we first recall multiple polylogarithms of $\ast$-type and of $\sh$-type in several variables defined by
\begin{align}
\Li_{s_1,\ldots,s_r}^{\ast}(z_1,\ldots,z_r)&=\sum_{1\leq m_1<\cdots<m_r} \frac{z_1^{m_1}\cdots z_r^{m_r}}{m_1^{s_1}m_2^{s_2}\cdots m_r^{s_r}}, \label{5-0}\\
\Li_{s_1,\ldots,s_r}^{\sh}(z_1,\ldots,z_r)&=\sum_{1\leq m_1<\cdots<m_r} \frac{z_1^{m_1}z_2^{m_2-m_1}
\cdots z_r^{m_r-m_{r-1}}}{m_1^{s_1}m_2^{s_2}\cdots m_r^{s_r}}\label{5-1-2}\\
& =\sum_{l_1,\ldots,l_r=1}^\infty \frac{z_1^{l_1}z_2^{l_2}\cdots z_r^{l_r}}{l_1^{s_1}(l_1+l_2)^{s_2}\cdots (l_1+\cdots+l_r)^{s_r}} \notag
\end{align}
for $s_1,\ldots,s_r\in \mathbb{C}$ and $z_1,\cdots,z_r\in \mathbb{C}$ with $|z_j|\leq 1$ $(1\leq j\leq r)$ (see, for example, \cite{Gon}). 
The symbols $\ast$ and $\sh$ are derived from the harmonic product and the shuffle product in the theory of multiple zeta values. 
In fact, Arakawa and the first-named author defined the two types of multiple $L$-values 
$L^{\ast}(k_1,\ldots,k_r;f_1,\ldots,f_r)$ of $\ast$-type and $L^{\sh}(k_1,\ldots,k_r;f_1,\ldots,f_r)$ of 
$\sh$-type associated to periodic functions $\{f_j\}$ (see \cite{AK2004}), 
defined by replacing $\{z_j^{m}\}$ by $\{f_j(m)\}$ and setting $(s_j)=(k_j)\in \mathbb{Z}_{\geq 1}^r$ on the right-hand sides 
of \eqref{5-0} and \eqref{5-1-2} for $(\kk)\in \mathbb{Z}_{\geq 1}^r$. Note that
\begin{equation}
\Li_{s_1,\ldots,s_r}^{\ast}(z_1,\ldots,z_r)=\Li_{s_1,\ldots,s_r}^{\sh}(\prod_{j=1}^{r}z_j,\prod_{j=2}^{r}z_j,\ldots,z_{r-1}z_{r},z_r). \label{5-1-3}
\end{equation}

\begin{definition}[Multi-indexed poly-Bernoulli numbers]\label{Def-5-1} \ 
For $s_1,\ldots,s_r\in \mathbb{C}$ and $\aa \in  \{1,2,\ldots,r\}$, the multi-indexed poly-Bernoulli numbers 
$\{ \bm_{\mm}^{(s_1,s_2,\ldots,s_r),(\aa)} \}$ are defined by 
\begin{align}
F(x_1,\ldots,x_r;s_1,\ldots,s_r;\aa)&=\frac{\Li_{s_1,\cdots,s_r}^{\sh}\left(1-e^{-\sum_{\nu=1}^{r}x_\nu},
\ldots,1-e^{-x_{r-1}-x_r}, 1-e^{-x_r}\right)}{\prod_{j=1}^{\aa}\left(1-e^{-\sum_{\nu=j}^rx_\nu}\right)}\label{5-2}\\
& \left(=\sum_{l_1,\ldots,l_r=1}^\infty \frac{\prod_{j=1}^{r}
\left(1-e^{-\sum_{\nu=j}^rx_\nu}\right)^{l_j-\delta_j(\aa)}}{\prod_{j=1}^{r}
\left(\sum_{\nu=1}^{j}l_\nu\right)^{s_j}}\right) \notag\\
& =\sum_{m_1,\ldots,m_r=0}^\infty \bm_{m_1,\ldots,m_r}^{(s_1,\ldots,s_r),(\aa)} 
\frac{x_1^{m_1}\cdots x_r^{m_r}}{m_1!\cdots m_r!},  \notag
\end{align}
where $\delta_j(\aa)=1\ (j\leq \aa)$, $=0\ (j> \aa)$.
\end{definition}

\begin{remark}\label{Rem-5-1}
Note that $\Li_{k_1,\ldots,k_r}^{\sh}(z,\ldots,z)=\Li_{k_1,\ldots,k_r}(z)$ defined by \eqref{1-7}. 
Suppose $x_1=\cdots=x_{r-1}=0$ and $(s_j)=(k_j)\in \mathbb{Z}^r$ in \eqref{5-2}. 
We immediately see that if $d=1$ then 
$$\bm_{0,\ldots,0,m}^{(k_1,\ldots,k_r),(1)}={B}_{m}^{(k_1,\ldots,k_r)}\qquad (m\in \mathbb{Z}_{\geq 0})$$
(see \eqref{1-4}), and if $\aa=r$ then 
$$\bm_{0,\ldots,0,m}^{(k_1,\ldots,k_r),(r)}=\mathbb{B}_{m}^{(k_1,\ldots,k_r)}\qquad (m\in \mathbb{Z}_{\geq 0})$$
(see \eqref{1-5-2}).
\end{remark}

\begin{remark}\label{Rem-5-2}
Let 
\begin{equation}
\Lambda_r=\{(x_1,\ldots,x_r)\in \mathbb{C}^{r}\mid |1-e^{-\sum_{\nu=j}^{r}x_\nu}|<1\ (1\leq j\leq r)\} . \label{region}
\end{equation}
Then we can see that 
$$\Li_{s_1,\cdots,s_r}^{\sh}\left(1-e^{-\sum_{\nu=1}^{r}x_\nu},\ldots,1-e^{-x_{r-1}-x_r}, 1-e^{-x_r}\right)
\quad (s_1,\ldots,s_r\in \mathbb{C})$$
is absolutely convergent for $(x_j)\in \Lambda_r$. Also $F(x_1,\ldots,x_r;s_1,\ldots,s_r;d)$ is absolutely convergent in the region $\Lambda_r\times \mathbb{C}^r$, 
so is holomorphic. 
Hence 
$\bm_{m_1,\ldots,m_r}^{(s_1,\ldots,s_r),(\aa)}$ is an entire function, 
because 
$$\bm_{m_1,\ldots,m_r}^{(s_1,\ldots,s_r),(\aa)} =\left(\frac{\partial}{\partial x_1}\right)^{m_1}
\cdots \left(\frac{\partial}{\partial x_r}\right)^{m_r}F(x_1,\ldots,x_r;s_1,\ldots,s_r;d)\bigg|_{(x_1,\ldots,x_r)=(0,\ldots,0)}$$
is holomorphic for all $(s_1,\ldots,s_r)\in \mathbb{C}^r$. 
\end{remark}

In the preceding section, we gave a certain duality formula for ${B}_{m}^{(k_1,\ldots,k_r)}$ (see Theorem \ref{Th-4-6}). 
By the similar method, we can prove certain duality formulas for $\bm_{m_1,\ldots,m_r}^{(k_1,\ldots,k_r),(\aa)} $, 
though they may be complicated. Hence, in the rest of this section, we will consider the case $\aa=r$. 
For emphasis, we denote $\bm_{m_1,\ldots,m_r}^{(s_1,\ldots,s_r), (r)}$ by $\bbm_{m_1,\ldots,m_r}^{(s_1,\ldots,s_r)}$. 
Note that $\delta_j(r)=1$ for any $j$. With this notation, we prove the following duality formulas.

\begin{theorem}\label{T-5-8}\ 
For $m_1,\ldots,m_r,\,\kk\in \mathbb{Z}_{\geq 0}$,
\begin{equation}
\bbm_{\mm}^{(\km)}=\bbm_{\kk}^{(-m_1,\ldots,-m_r)}. \label{5-18}
\end{equation}
\end{theorem}

Now we aim to prove this theorem. First we generalize Lemma \ref{L-4-1} as follows. 

\begin{lemma}\label{L-5-4}
For $\kk\in \mathbb{Z}_{\geq 0}$,  there exists a polynomial $\ppp(x_1,\ldots,x_r;\kk)\in \mathbb{Z}[x_1,\ldots,x_r]$ 
such that
\begin{align}
& \Li_{\km}^{\ast}(z_1,\ldots,z_r)=\frac{\ppp(\prod_{j=1}^{r}z_j,\prod_{j=2}^{r}z_j,
\ldots,z_{r-1}z_r,z_r;\kk)}{\prod_{j=1}^{r}\left(1-\prod_{\nu=j}^{r}z_\nu \right)^{\sum_{\nu=j}^{r}k_\nu+1}}, \label{5-4}\\
& {\rm deg}_{x_j}\ppp(x_1,\ldots,x_r;\kk)\leq \sum_{\nu=j}^{r}k_\nu+1, \label{5-5}\\
& (x_1\cdots x_r)\mid \ppp(x_1,\ldots,x_r;\kk).\label{5-6}
\end{align}
Set $y_j=\prod_{\nu=j}^{r}z_\nu$ $(1\leq j\leq r)$. Then \eqref{5-4} implies
\begin{align}
& \Li_{\km}^{\sh}(y_1,\ldots,y_r)=\frac{\ppp(y_1,\cdots,y_r;\kk)}{\prod_{j=1}^{r}
\left(1-y_j\right)^{\sum_{\nu=j}^{r}k_\nu+1}}. \label{5-7}
\end{align}
\end{lemma}

\begin{proof}
In order to prove this lemma, we have only to use the same method as in Lemma \ref{L-4-1} by induction on $r$. 
Since the case of $r=1$ is proven, we consider the case of $r\geq 2$. Further, when $K=k_1+\cdots+k_r=0$, 
it is easy to have the assertion. Hence we think about a general case $K=k_1+\cdots+k_r(\geq 1)$. 
When $k_r=0$, we have
\begin{align*}
\Li_{\km}^{\ast}(z_1,\ldots,z_r)&=\frac{z_r}{1-z_r} \Li_{-k_1,\ldots,-k_{r-1}}^{\ast}(z_1,\ldots,z_{r-2},z_{r-1}z_r)\\
& =\frac{z_r}{1-z_r}\frac{\ppp(\prod_{j=1}^{r}z_j,\ldots,z_{r-1}z_r;k_1,
\ldots,k_{r-1})}{\prod_{j=1}^{r-1}\left(1-\prod_{\nu=j}^{r}z_j\right)^{\sum_{\nu=j}^{r}k_\nu+1}}.
\end{align*}
Therefore, setting $\ppp(x_1,\cdots,x_r;k_1,\ldots,k_{r-1},0)=x_r\ppp(x_1,\ldots,x_{r-1};k_1,\ldots,k_{r-1})$, 
we can verify \eqref{5-4}--\eqref{5-6}.

Next we consider the case $k_r\geq 1$. For $k\in \mathbb{Z}_{\geq 0}$, we inductively define a subset 
$\{c_{j,\nu}^{(k)}\}_{0\leq j,\nu\leq k+1}$
of $\mathbb{Z}$ by
\begin{equation}
\frac{d}{dz}\left(\sum_{m>l}m^{k}z^m\right)=\frac{1}{(1-z)^{k+2}}\sum_{j=0}^{k+1}
\sum_{\nu=0}^{k+1}c_{j,\nu}^{(k)}l^\nu z^{l+j}. \label{5-8}
\end{equation}
In fact, by
$$\frac{d}{dz}\left(\sum_{m>l}z^m\right)=\frac{1}{(1-z)^{2}}\left(z^l+l z^l-l z^{l+1}\right),$$
and 
$$\sum_{m>l}m^{k}z^m=z\frac{d}{dz}\left(\sum_{m>l}m^{k-1}z^m\right)\qquad (k\geq 1),$$
we can determine $\{c_{j,\nu}^{(k)}\}$ by \eqref{5-8}. Using this notation, we have
\begin{align*}
& \Li_{\km}^{\ast}(z_1,\ldots,z_r)=z_r\frac{d}{dz_r}\Li_{-k_1,\ldots,-k_r+1}^{\ast}(z_1,\ldots,z_r)\\
& =z_r \sum_{m_1<\cdots<m_{r-1}}m_1^{k_1}\cdots m_{r-1}^{k_{r-1}}z_1^{m_1}\cdots z_{r-1}^{m_{r-1}}\frac{\sum_{j=0}^{k_r}\sum_{\nu=0}^{k_r}c_{j,\nu}^{(k_r-1)}m_{r-1}^\nu z_r^{m_{r-1}+j}}{(1-z_r)^{k_r+1}}\\
& =\frac{1}{(1-z_r)^{k_r+1}} \sum_{j=0}^{k_r}\sum_{\nu=0}^{k_r}c_{j,\nu}^{(k_r-1)}z_r^{j+1}\sum_{m_1<\cdots<m_{r-1}}m_1^{k_1}\cdots m_{r-1}^{k_{r-1}+\nu} z_1^{m_1}\cdots (z_{r-1}z_r)^{m_{r-1}}.
\end{align*}
By the induction hypothesis in the case $r-1$, this is equal to 
\begin{align*}
& \frac{1}{(1-z_r)^{k_r+1}} \sum_{j=0}^{k_r}\sum_{\nu=0}^{k_r}c_{j,\nu}^{(k_r-1)}z_r^{j+1}\frac{\ppp(\prod_{j=1}^{r-1}z_j,\ldots,z_{r-1}z_r;k_1,\ldots,k_{r-2},k_{r-1}+\nu)}{\prod_{j=1}^{r-2}
\left(1-\prod_{\nu=j}^{r}z_j\right)^{\sum_{\nu=j}^{r}k_\nu+1}(1-z_{r-1}z_{r})^{k_{r-1}+\nu+1}}.
\end{align*}
Therefore we set
\begin{align*}
&\ppp(x_1,\cdots,x_r;\kk)\\
&=\sum_{j=0}^{k_r}\sum_{\nu=0}^{k_r}c_{j,\nu}^{(k_r-1)}x_r^{j+1}(1-x_{r-1})^{k_r-\nu}\ppp(x_1,\cdots,x_{r-1};k_1,\ldots,k_{r-2},k_{r-1}+\nu).
\end{align*}
Then this satisfies \eqref{5-4}--\eqref{5-6}. This completes the proof.
\end{proof}

From this result, we can reach the following definition.

\begin{definition}\label{Def-Main-3} 
For $k_1,\ldots,k_r\in \mathbb{Z}_{\geq 0}$, define 
\begin{equation}
\begin{split}
& \eett(\km;s_1,\ldots,s_r) =\frac{1}{\prod_{j=1}^{r}\Gamma(s_j)}\int_{0}^\infty \cdots \int_{0}^\infty 
\prod_{j=1}^{r} t_j^{s_j-1}\\
& \qquad \times \frac{\Li_{\km}^\sh(1-e^{\sum_{\nu=1}^{r}t_\nu},
\ldots,1-e^{t_{r-1}+t_r},1-e^{t_r})}{\prod_{j=1}^{r}(1-e^{\sum_{\nu=j}^{r}t_\nu}) }\prod_{j=1}^{r}dt_j
\end{split}
\label{5-9}
\end{equation}
for $s_1,\ldots, s_r\in \mathbb{C}$ with ${\rm Re}(s_j)>0$ $(1\leq j\leq r)$. 
\end{definition}

Lemma \ref{L-5-4} ensures 
that the integral on the right-hand side of \eqref{5-9} is absolutely convergent for ${\rm Re}(s_j)>0$. 
By the same method as in the proof of Theorem \ref{Th-Main-1} for ${\eta}(\kk;s)$, we can similarly obtain the following.

\begin{theorem}\label{Th-Main-3} 
For $k_1,\ldots,k_r\in \mathbb{Z}_{\geq 0}$, 
$\eett(\km;s_1,\ldots,s_r)$ can be analytically continued to an entire function on the whole complex space, and satisfies 
\begin{equation}
\eett(\km;-m_1,\ldots,-m_r)=\bbm_{m_1,\ldots,m_r}^{(\km)}\quad (m_1,\ldots,m_r \in \mathbb{Z}_{\geq 0}). \label{5-10}
\end{equation}
\end{theorem}

\begin{proof}
As in the proof of Theorem \ref{Th-Main-1}, let
\begin{align}
&H(\km;s_1,\ldots,s_r) \label{5-11}\\
&=\int_{\mathcal{C}^r}\prod_{j=1}^{r} t_j^{s_j-1}\frac{\Li_{\km}^\sh (1-e^{\sum_{\nu=1}^{r}t_\nu},
\ldots, 1-e^{t_r})}{\prod_{j=1}^{r}(1-e^{\sum_{\nu=j}^{r}t_\nu}) }\prod_{j=1}^{r}dt_j\notag\\
& =\prod_{j=1}^{r}(e^{2\pi i s_j}-1)\int_{\varepsilon}^\infty\cdots \int_{\varepsilon}^\infty \prod_{j=1}^{r}t_j^{s_j-1}\frac{\Li_{\km}^\sh (1-e^{\sum_{\nu=1}^{r}t_\nu},\ldots, 1-e^{t_r})}{\prod_{j=1}^{r}(1-e^{\sum_{\nu=j}^{r}t_\nu}) }\prod_{j=1}^{r}dt_j\notag\\
& \qquad \cdots +\int_{C_\varepsilon^r}\prod_{j=1}^{r}t_j^{s_j-1} \frac{\Li_{\km}^\sh (1-e^{\sum_{\nu=1}^{r}t_\nu},\ldots, 1-e^{t_r})}{\prod_{j=1}^{r}(1-e^{\sum_{\nu=j}^{r}t_\nu}) }\prod_{j=1}^{r}dt_j,\notag
\end{align}
where $\mathcal{C}^r$ is the direct product of the contour $\mathcal{C}$ defined before. 
Note that the integrand on the second member has no singularity on $\mathcal{C}^r$. 
It follows from Lemma \ref{L-5-4} that $H(\km;s_1,\ldots,s_r)$ is absolutely convergent for any 
$(s_j)\in \mathbb{C}^r$, namely is entire. 
Suppose ${\rm Re}(s_j)>0$ for each $j$, all terms except for the first term on the third member of \eqref{5-11} tends to $0$ as $\varepsilon \to 0$. Hence
$$\eett(\km;s_1,\ldots,s_r)=\frac{1}{\prod_{j=1}^{r}(e^{2\pi i s_j}-1)\Gamma(s_j)}H(\km;s_1,\ldots,s_r),$$
which can be analytically continued to $\mathbb{C}^r$. 
Also, setting $(s_1,\ldots,s_r)=(-m_1,\ldots,-m_r) \in \mathbb{Z}_{\leq 0}^r$ in \eqref{5-11}, 
we obtain  \eqref{5-10} from \eqref{5-2}. 
This completes the proof.
\end{proof}

Next we directly construct the generating function of $\eett(\km;s_1,\ldots,s_r)$. We prepare the following 
two lemmas which we consider when $(x_j)$ is in $\Lambda_r$ defined by \eqref{region}.

\begin{lemma}\label{L-5-3}\ \ For $(s_j)\in \mathbb{C}^r$ with ${\rm Re}(s_j)>0$ $(1\leq j\leq r)$, 
\begin{align}
& F(x_1,\ldots,x_r;s_1,\ldots,s_r;r)\label{5-13}\\
& =\frac{1}{\prod_{j=1}^{r}\Gamma(s_j)}\int_0^\infty\cdots\int_0^\infty \prod_{j=1}^{r}
\left\{t_j^{s_j-1}\frac{e^{\sum_{\nu=j}^{r}x_\nu}}{1-e^{\sum_{\nu=j}^{r}x_\nu}
(1-e^{\sum_{\nu=j}^{r}t_\nu})}\right\}\prod_{j=1}^{r}dt_j.\notag
\end{align}
\end{lemma}

\begin{proof}
Substituting $n^{-s}=(1/\Gamma(s))\int_{0}^\infty t^{s-1}e^{-nt}dt$ into the second member of \eqref{5-2}, we have
\begin{align*}
F(\{x_j\};\{s_j\};r)&=\sum_{l_1,\ldots,l_r=1}^\infty \prod_{j=1}^{r}
\left(1-e^{-\sum_{\nu=j}^r x_\nu}\right)^{l_j-1}\frac{1}{\prod_{j=1}^{r}\Gamma(s_j)}\\
& \times \int_0^\infty\cdots\int_0^\infty \prod_{j=1}^{r}\left\{t_j^{s_j-1}\exp\left(-(\sum_{\nu=1}^{j}l_\nu)t_j\right)\right\}\prod_{j=1}^{r}dt_j.
\end{align*}
We see that the integrand on the right-hand side can be rewritten as 
$$\prod_{j=1}^{r}t_j^{s_j-1}\prod_{j=1}^{r}\exp\left(-l_j(\sum_{\nu=j}^{r}t_\nu)\right).$$
Hence we have 
\begin{align*}
F(\{x_j\};\{s_j\};r)&=\frac{1}{\prod_{j=1}^{r}\Gamma(s_j)\left(1-e^{-\sum_{\nu=j}^r x_\nu}\right)}
\int_0^\infty\cdots\int_0^\infty \prod_{j=1}^{r}t_j^{s_j-1}\\
& \times \sum_{l_1,\ldots,l_r=1}^\infty \prod_{j=1}^{r}\left(1-e^{-\sum_{\nu=j}^r x_\nu}\right)^{l_j}
e^{-l_j(\sum_{\nu=j}^{r}t_\nu)}\prod_{j=1}^{r}dt_j\\
&=\frac{1}{\prod_{j=1}^{r}\Gamma(s_j)}\int_0^\infty\cdots\int_0^\infty \prod_{j=1}^{r}t_j^{s_j-1}
\frac{e^{-\sum_{\nu=j}^{r}t_\nu}}{1-(1-e^{-\sum_{\nu=j}^{r}x_\nu})e^{-\sum_{\nu=j}^{r}t_\nu}}\prod_{j=1}^{r}dt_j\\
&=\frac{1}{\prod_{j=1}^{r}\Gamma(s_j)}\int_0^\infty\cdots\int_0^\infty \prod_{j=1}^{r}t_j^{s_j-1}
\frac{e^{\sum_{\nu=j}^{r}x_\nu}}{1-e^{\sum_{\nu=j}^{r}x_\nu}(1-e^{\sum_{\nu=j}^{r}t_\nu})}\prod_{j=1}^{r}dt_j.
\end{align*}
This completes the proof.
\end{proof}

\begin{lemma}\label{L-5-5}\ 
Let $z_1,\ldots,z_r\in \mathbb{C}$ and assume that $|z_j|$ $(1\leq j\leq r)$ are sufficiently small. Then
\begin{align}
& \prod_{j=1}^{r}\frac{z_je^{\sum_{\nu=j}^{r}x_\nu}}{1-z_je^{\sum_{\nu=j}^{r}x_\nu}}=\sum_{\kk=0}^\infty \Li_{\km}^\sh(z_1,\ldots,z_r)\frac{x_1^{k_1}\cdots x_r^{k_r}}{k_1!\cdots k_r!}. \label{5-14-0}
\end{align}
Set $z_j=1-e^{\sum_{\nu=j}^{r}t_\nu}$ $(1\leq j\leq r)$ for 
$(t_j)\in \Lambda_r$. Then
\begin{align}
& \prod_{j=1}^{r}\frac{e^{\sum_{\nu=j}^{r}x_\nu}\left(1-e^{\sum_{\nu=j}^{r}t_\nu}\right)}{1-e^{\sum_{\nu=j}^{r}x_\nu}\left(1-e^{\sum_{\nu=j}^{r}t_\nu}\right)} \label{5-14}\\
& \quad =\sum_{\kk=0}^\infty \Li_{\km}^\sh(1-e^{\sum_{\nu=1}^{r}t_\nu},\ldots,1-e^{t_r})\frac{x_1^{k_1}
\cdots x_r^{k_r}}{k_1!\cdots k_r!}.\notag
\end{align}
In particular, the case $t_1=\cdots=t_{r-1}=0$ and $t_r=t$ implies \eqref{4-10}.
\end{lemma}

\begin{proof}
We have only to prove \eqref{5-14-0}. Actually we have 
\begin{align*}
& \sum_{\kk=0}^\infty \Li_{\km}^\sh(z_1,\ldots,z_r)\frac{x_1^{k_1}\cdots x_r^{k_r}}{k_1!\cdots k_r!}\\
& =\sum_{\kk=0}^\infty\sum_{m_1,\ldots,m_r=1}^\infty  \prod_{j=1}^{r}
\frac{((\sum_{\mu=1}^{j}m_\mu)x_j)^{k_j}}{k_j!}z^{m_j}\\
& =\sum_{m_1,\ldots,m_r=1}^\infty  \prod_{j=1}^{r} z_j^{m_j}\prod_{\mu=1}^{j}e^{m_\mu x_j}\\
& =\sum_{m_1,\ldots,m_r=1}^\infty  \prod_{j=1}^{r} \left(z_je^{\sum_{\nu=j}^{r}x_\nu}\right)^{m_j}=\prod_{j=1}^{r}\frac{z_je^{\sum_{\nu=j}^{r}x_\nu}}{1-z_je^{\sum_{\nu=j}^{r}x_\nu}}.
\end{align*}
Thus we have the assertion.
\end{proof}

Using these lemmas, we obtain the following. 

\begin{theorem}\label{T-5-9}
For $k_1,\ldots,k_r\in \mathbb{Z}_{\geq 0}$, 
\begin{equation}
\eett(\km;s_1,\ldots,s_r)=\bbm_{\kk}^{(s_1,\ldots,s_r)}. \label{5-17}
\end{equation}
\end{theorem}

\begin{proof}
By Lemmas \ref{L-5-3} and \ref{L-5-5}, we have
\begin{align}
& F(x_1,\ldots,x_r;s_1,\ldots,s_r;r)\label{5-19}\\
& =\frac{1}{\prod_{j=1}^{r}\Gamma(s_j)}\int_0^\infty\cdots\int_0^\infty 
\prod_{j=1}^{r}\left\{t_j^{s_j-1}\frac{e^{\sum_{\nu=j}^{r}x_\nu}}{1-e^{\sum_{\nu=j}^{r}x_\nu}
(1-e^{\sum_{\nu=j}^{r}t_\nu})}\right\}\prod_{j=1}^{r}dt_j\notag\\
& =\frac{1}{\prod_{j=1}^{r}\Gamma(s_j)}\sum_{\kk=0}^\infty 
\bigg\{\int_0^\infty\cdots\int_0^\infty \prod_{j=1}^{r}t_j^{s_j-1}\notag\\
& \qquad \times \frac{\Li_{\km}^\sh(1-e^{\sum_{\nu=1}^{r}t_\nu},
\ldots,1-e^{t_r})}{\prod_{j=1}^{r}\left(1-e^{\sum_{\nu=j}^{r}t_\nu}\right)}
\prod_{j=1}^{r}dt_j\bigg\} \frac{x_1^{k_1}\cdots x_r^{k_r}}{k_1!\cdots k_r!} \notag
\end{align}
for ${\rm Re}(s_j)>0$ $(1\leq j\leq r)$. Combining \eqref{5-2}, \eqref{5-9} and \eqref{5-19}, 
we obtain \eqref{5-17} for ${\rm Re}(s_j)>0$ $(1\leq j\leq r)$, hence for all $(s_j)\in \mathbb{C}$, 
because both sides of \eqref{5-17} are entire functions (see Remark \ref{Rem-5-2}). 
\end{proof}

\begin{proof}[Proof of Theorem \ref{T-5-8}]
Setting $(s_1,\ldots,s_r)=(-m_1,\ldots,-m_r)$ in \eqref{5-17}, we obtain \eqref{5-18} from \eqref{5-10}. 
This completes the proof of Theorem \ref{T-5-8}.
\end{proof}

\begin{example}
We can easily see that
$$\Li_{-1,0}^\sh(z_1,z_2)=\frac{z_1z_2}{(1-z_1)^2(1-z_2)},\ \ \Li_{0,-1}^\sh(z_1,z_2)=
\frac{z_1z_2(2-z_1-z_2)}{(1-z_1)^2(1-z_2)^2}.$$
Hence we have
$$\bbm_{m,n}^{(-1,0)}=2^m3^n,\quad \bbm_{m,n}^{(0,-1)}=(2^m+1)3^n\quad (m,n\in \mathbb{Z}_{\geq 0}). $$
Therefore $\bbm_{0,1}^{(-1,0)}=\bbm_{1,0}^{(0,-1)}=3.$ Similarly we obtain, for example,
\begin{align*}
& \bbm_{1,0}^{(-1,-2)}=\bbm_{1,2}^{(-1,0)}=18, \ \bbm_{1,2}^{(-3,-1)}=\bbm_{3,1}^{(-1,-2)}=1820, \ \bbm_{2,2}^{(-2,-1)}=\bbm_{2,1}^{(-2,-2)}=1958.
\end{align*}

\end{example}

\begin{remark}
Hamahata and Masubuchi \cite[Corollary 10]{HM2007-1} showed the special case of \eqref{5-18}, namely 
$$\bbm_{0,\ldots,0,m}^{(0,\ldots,0,-k)}=\bbm_{0,\ldots,0,k}^{(0,\ldots,0,-m)}\quad (m,k\in \mathbb{Z}_{\geq 0})$$
(see Remark \ref{Rem-5-1}). On the other hand, Theorem \ref{Th-4-6} corresponds to the case 
$\aa=1\neq r$ except for $r=1$ (see Remark \ref{Rem-5-1}), hence is located in the outside of 
Theorem \ref{T-5-8}. Therefore, in \eqref{4-9}, another type of multi-poly-Bernoulli numbers appear. 
\end{remark}

\ 

{\bf Acknowledgments.} \ 
The authors would like to thank Yasushi Komori for pointing out a mistake in an earlier version
of this paper.

\

\ 

\ 

\begin{flushleft}
\begin{small}
{M. Kaneko}: 
{Faculty of Mathematics, 
Kyushu University, 
Motooka 744, Nishi-ku
Fukuoka 819-0395, 
Japan}

e-mail: {\tt mkaneko@math.kyushu-u.ac.jp}

\

{H. Tsumura}: 
{Department of Mathematics and Information Sciences, Tokyo Metropolitan University, 1-1, Minami-Ohsawa, Hachioji, Tokyo 192-0397 Japan}

e-mail: {\tt tsumura@tmu.ac.jp}
\end{small}
\end{flushleft}

\end{document}